\documentclass[12pt]{amsart}

\usepackage{amssymb,euscript,amsmath, mathrsfs,amscd}
\usepackage{epsfig} %,pstricks}
\usepackage{enumerate}
\usepackage{fullpage}
\usepackage{tikz}
\newcounter{ENUM}
\newcommand{\beas}{\begin{eqnarray*}}
\newcommand{\eeas}{\end{eqnarray*}}
\newcommand{\itm}{\item}
\newenvironment{ilist}{\renewcommand{\theENUM}{\roman{ENUM}}\renewcommand{\itm}{\addtocounter{ENUM}{1}\item[(\theENUM)]}\begin{itemize}\setcounter{ENUM}{0}}{\end{itemize}}
\newenvironment{alist}[1][0]{\renewcommand{\theENUM}{\alph{ENUM}}\renewcommand{\itm}{\addtocounter{ENUM}{1}\item[\theENUM)]}\begin{itemize}\setcounter{ENUM}{#1}}{\end{itemize}}

\newtheorem{thm}{Theorem}[section]
\newtheorem{prop}[thm]{Proposition}
\newtheorem{lem}[thm]{Lemma}

\newtheorem{cor}[thm]{Corollary}
\newtheorem{conj}[thm]{Conjecture}

\theoremstyle{definition}
\newtheorem{defn}[thm]{Definition}
\newtheorem{setup}[thm]{Setup}

\newtheorem{ex}[thm]{Example}

\theoremstyle{remark}
\newtheorem{notn}[thm]{Notation}
\newtheorem{rem}[thm]{Remark}

\numberwithin{equation}{section}

\def\Z{{\mathbb Z}}
\def\R{{\mathbb R}}
\def\fS{{\mathfrak S}}

\def\ds{\displaystyle}
\newcommand{\bm}[1]{{\boldsymbol{#1}}}
\def\0{\bm{0}}

\def\c{\bm{c}}
\def\e{\bm{e}}

\def\m{\bm{m}}

\def\p{\bm{p}}

\def\v{\bm{v}}
\def\u{\bm{u}}

\def\x{\bm{x}}
\def\y{\bm{y}}

\def\bbeta{\bm{\beta}}

\def\cC{\mathcal C}

\def\cM{\mathcal M}

\def\cS{\mathcal S}
\def\O{\mathcal O}
\def\cP{\mathcal P}

\newcommand{\conv}{\mathrm{conv}}

\newcommand{\Lat}{\mathrm{Lat}}
\newcommand{\vol}{\mathrm{Vol}}
\newcommand{\nvol}{\mathrm{nvol}}

\newcommand{\mv}{\mathrm{\mathcal{M}Vol}}
\newcommand{\ml}{\mathrm{\mathcal{M}Lat}}
\newcommand{\mphi}{\mathcal{M}\phi}

\def\x{\mathbf{x}}

\def\ncone{\operatorname{ncone}}
\def\conv{\operatorname{conv}}
\def\cone{\operatorname{Cone}}

\def\Perm{\operatorname{Perm}}

\def\aff{\operatorname{aff}}
\def\lin{\operatorname{lin}}

\newcommand{\B}{{\mathfrak B}}

\author{Federico Castillo}
\address{Federico Castillo, Department of Mathematics, University of
  California, Davis, One Shields Avenue, Davis, CA 95616 USA.}
  \email{fcastillo@math.ucdavis.edu}
  
  \author{Fu Liu}
\thanks{Fu Liu is partially supported by NSF grant DMS-1265702 and a grant from from the Simons Foundation \#426756.} \address{Fu Liu, Department of Mathematics, University of California, Davis, One Shields Avenue, Davis, CA 95616 USA.}
\email{fuliu@math.ucdavis.edu}

\begin{document}
\title{Berline-Vergne valuation and generalized permutohedra}

\maketitle
\begin{abstract}
Generalizing a conjecture by De Loera et al., we conjecture that integral generalized permutohedra all have positive Ehrhart coefficients. Berline and Vergne construct a valuation that assigns values to faces of polytopes, which provides a way to write Ehrhart coefficients of a polytope as positive sums of these values. Based on available results, we pose a stronger conjecture: Berline-Vergne's valuation is always positive on permutohedra, which implies our first conjecture. 

This article proves that our strong conjecture on Berline-Vergne's valuation is true for dimension up to $6$, and is true if we restrict to faces of codimension up to $3.$ 
In addition to investigating the positivity conjectures, we study the Berline-Vergne's valuation, and show that it is the unique construction for McMullen's formula used to describe number of lattice points in permutohedra under certain symmetry constraints.
 We also give an equivalent statement to the strong conjecture in terms of mixed valuations. % and Todd class, respectively. 

\end{abstract}

\section{Introduction}
\label{sec:in}

A \emph{lattice point} is a point in $\Z^D.$ Given any bounded set $S \subseteq \R^D,$ we let $\Lat(S) := |S \cap \Z^D|$ be the number of lattice points in $S.$
%Let $P\subset \mathbb{R}^n$ be a \emph{lattice polytope}, i.e., a polytope whose vertices are all lattice points. 
Given a polytope $P$ in $\R^D,$ a natural enumerative problem is to compute $\Lat(P)$. 
In this paper, we will focus on \emph{integral polytopes}, i.e., polytopes whose vertices are all lattice points, and generalized permutohedra -- a special family of polytopes.

\subsection{Motivation: Ehrhart positivity for generalized permutohedra}
One approach to study the question of computing $\Lat(P)$ for an integral polytope $P$ is to consider a more general counting problem: For any nonnegative integer $t,$ let $tP := \{ t \x \ | \ \x \in P\}$ be the \emph{$t$-th dilation of $P$}, and then consider the function 
\[
i(P, t) := \Lat(tP)
\]
that counts the number of lattice points in $tP.$
It is a classic result that $i(P,t)$ is a polynomial in $t$. More precisely:
\begin{thm}[Ehrhart \cite{ehrhart}]\label{thm:ehrhart0}
There exists a polynomial $f(x)$ such that $f(t) = i(P,t)$ for any $t \in \mathbb{Z}_{\geq 0}$. 
Moreover, the degree of $f(x)$ is equal to the dimension of $P$.
\end{thm}
%Its degree is equal to the dimension of $P$ and 
We call the function $i(P, t)$ the \emph{Ehrhart polynomial} of $P.$
A few coefficients of $i(P,t)$ are well understood: the leading coefficient is equal to the normalized volume of $P$, the second coefficient is one half of the sum of the normalized volumes of facets, and the constant term is always $1$. Although formulas are derived for the other coefficients, they are quite complicated. One notices that the leading, second and last coefficients of the Ehrhart polynomial of any integral polytope are always positive; but it is not true for the rest of the coefficients for general polytopes. We say a polytope has \emph{Ehrhart positivity} or is \emph{Ehrhart positive} if it has positive Ehrhart coefficients. %In fact, there are few families of polytopes known to have positive Ehrhart coefficients, a property we call \emph{Ehrhart positivity}. 

A few families of polytopes are known to be Ehrhart positive. Zonotopes, in particular regular permutohedra, are Ehrhart positive \cite[Theorem 2.2]{zonotopes}. Cyclic polytopes also have this property. Their Ehrhart coefficients are given by the volumes of certain projections of the original polytope \cite{cyclic}. Stanley-Pitman polytopes are defined in \cite{stanley-pitman} where a formula for their Ehrhart polynomials is given and from which Ehrhart positivity follows. Recently in \cite{deloera} De Loera, Haws, and Koeppe study the case of matroid base polytopes and conjecture they are Ehrhart positive. Both Stanley-Pitman polytopes and matroid base polytopes fit into a bigger family: generalized permutohedra. 

In \cite{bible} Postnikov defines generalized permutohedra as polytopes obtained by moving the vertices of a usual permutohedron while keeping the same edge directions. %That's what he calls a generalized permutohedron of type $z$. 
He also considers a strictly smaller family, type $y$, consisting of sums of dilated simplices. He describes the Ehrhart polynomial for the type $y$ family in \cite[Theorem 11.3]{bible}, from which Ehrhart positivity follows. The type $y$ family includes the Stanley-Pitman polytopes, associahedra, cyclohedra, and more (see \cite[Section 8]{bible}), but fails to contain matroid base polytopes %, which are type $z$ generalized permutohedra 
\cite[Proposition 2.4]{matroidpolytope}.
%Our main conjecture is
We consider the following conjecture on generalized permutohedra, generalizing the conjecture on Ehrhart positivity of matroid base polytopes given in \cite{deloera} by De Loera et al,\begin{conj}\label{positivity}

Integral generalized permutohedra are Ehrhart positive.
\end{conj}
%Note that since generalized permutohedra contain the family of matroid base polytopes, our conjecture is a generalization of the conjecture on Ehrhart positivity of matroid base polytopes given in \cite{deloera} by De Loera et al.

The above conjecture is the original motivation of this paper. However,
instead of studying it directly, we reduce it to another conjecture which only concerns regular permutohedra, a smaller family of polytopes. 
%The approach we suggest consist of two parts. First of all we use an interpretation for the coefficients in terms of a valuation on the algebra of pointed cones introduced by Berline and Vergne \cite{localformula}. More precisely we use the exterior formula

\subsection{McMullen's formula, $\alpha$-positivity and uniqueness}
%In 1975 Danilov asked if it is possible to assign values $r_\sigma$ to cones $\sigma$ such that 
%\begin{equation}\label{eqn:danilov0}
%\td(X(\Delta)) = \displaystyle \sum_{\sigma\in\Delta} r_\sigma [V(\sigma)],
%\end{equation}
%where $\Delta$ is a complete fan, $X(\Delta)$ is the corresponding toric variety, and $V(\sigma)$ is the closed subvariety associated with $\sigma$.
%Notice the special feature of Formula \eqref{eqn:danilov0}: the value of $r_\sigma$ only depends on the cone $\sigma$, but not on the fan $\Delta.$ We call this the \emph{Danilov condition}. 

In 1975 Danilov \cite{danilov} asked, in the context of toric varieties, whether %to assign values $\Psi(C)$ to all rational cones $C$ such that 
the following formula holds for any integral polytope $P$:
\begin{equation}\label{equ:exterior}
\Lat(P) = \displaystyle \sum_{F: \textrm{ a face of $P$}} \alpha(F,P) \ \nvol(F) ,
\end{equation}
where the value of $\alpha(F,P)$ depends only on the normal cone of $P$ at $F$.
%where $\alpha(F,P)$ is to be the $\Psi$ value of the polar of the normal cone of $P$ at $F,$ and $\nvol(F)$ is the normalized volume of $F$. 

McMullen \cite{mcmullen} was the first to confirm the existence of Formula \eqref{equ:exterior} in a non constructive way. % (which was the reason we call this formula McMullen's formula).
Therefore, we call \eqref{equ:exterior} \emph{McMullen's formula}.
%Morelli \cite{morelli} was the first to supply an explicit and computationally efficient way to choose $\alpha(F, P)$. 
Morelli \cite{morelli} was the first to supply a deterministic algorithm to choose $\alpha(F, P)$. 
In \cite{toddclass}, Pommersheim and Thomas gave a canonical construction of $\alpha(F, P)$ from the Todd class
of a toric variety. In 2007, Berline and Vergne gave the first construction of $\alpha(F, P)$ in \cite{localformula} in the primary space without using toric varieties. 

%If such an expression exists, then the following \emph{McMullen's formula} holds for any integral polytope $P$:
%\begin{equation}\label{equ:exterior}
%\Lat(P) = \displaystyle \sum_{F: \textrm{ a face of $P$}} \alpha(F,P) \ \nvol(F) ,
%\end{equation}
%where $\alpha(F,P)$ is set to be $r_\sigma$ where $\sigma$ is the normal cone of $P$ at $F$, and $\nvol(F)$ is the normalized volume of $F$. 

%In fact, one can ask directly the existence of McMullen's formula (independently from Danilov's question). More specifically, one can ask whether there are ways to assign values to cones such that McMullen's formula holds if $\alpha(F,P)$ only depends on the normal cone of $P$ at $F.$ %We call this \emph{McMullen's question}. 
%We will discuss the implication $\eqref{eqn:danilov0} \Rightarrow \eqref{equ:exterior}$ in Section \ref{sec:toric}. But for most the paper, we focus on McMullen's formula.

%McMullen \cite{mcmullen} was the first to confirm the existence of Formula \eqref{equ:exterior} in a non constructive way (which was the reason we call this formula McMullen's formula). Morelli \cite{morelli} supplied the first explicit way to choose $r_\sigma$ answering Danilov's question. Pommersheim and Thomas \cite{toddclass} gave a canonical construction of $r_\sigma$ based on choices of flags. %$\alpha(P,F)$ from the Todd class of toric varieties.
%As we discussed above, both of these two constructions naturally give a way to construct $\alpha$ for McMullen's formula \eqref{equ:exterior}
%Berline and Vergne \cite{localformula} were able to construct such $\alpha$ in a computable way. 
One immediate consequence of the existence of Formula \eqref{equ:exterior} is that if $\alpha(F,P)$ is positive for each face $F$ of $P,$ then Ehrhart positivity follows. (See Theorem \ref{thm:ehrhart} and Lemma \ref{lem:red1}.) %This can be used to give an expression for each coefficient depending on volumes of faces (see Theorem \ref{ehrhart}), and the valuation applied to their tangent cones. As long as this valuation is positive in all such cones, all the coefficients will be positive. 
Hence, it is natural to say that a polytope $P$ has \emph{$\alpha$-positivity} or is \emph{$\alpha$-positive} if all $\alpha$'s associated to $P$ are positive.

Although there are different constructions for $\alpha(F,P)$, Berline-Vergne's construction has certain nice properties that are good for our purpose, and thus we will use their construction in our paper. We refer to their construction for $\alpha(F, P)$ as the \emph{BV-$\alpha$-valuation}. To be more precise, we will use the terminologies \emph{BV-$\alpha$-positivity} and \emph{BV-$\alpha$-positive} to indicate that the $\alpha$'s we use are from the BV-$\alpha$-valuation. 

At present, the explicit computation of the BV-$\alpha$-valuation is a recursive, complicated process, but we carry it out in the special example of regular permutohedra of small dimensions, whose symmetry simplifies the computations. Based on our empirical results, we conjecture the following:
\begin{conj}\label{conj:alphas}
%For any face $F$ of the regular permutohedron $\Pi_n$, we have $\alpha(F,\Pi_n)>0$.
  Every regular permutohedron is BV-$\alpha$-positive.
\end{conj}

One important property of the BV-$\alpha$-valuation enables us to reduce the problem of proving the Ehrhart positivity of all generalized permutohedra to proving the positivity of all the $\alpha$'s arising from the regular permutohedra.
\begin{thm}\label{thm:reduction}
  %If the regular permutohedron $\Pi_{n-1} \subset \R^n$ is BV-$\alpha$-positive, then every generalized permutohedron in $\R^n$ is BV-$\alpha$-positive.
  
Conjecture \ref{conj:alphas} implies Conjecture \ref{positivity}.
\end{thm}
Therefore, we focus on proving Conjecture \ref{conj:alphas} instead. In this paper, we provide partial progress on proving Conjecture \ref{conj:alphas} (and thus Conjecture \ref{positivity}):

\begin{thm}\label{thm:partial}
  
  %\label{thm:truelowdim}
  \begin{enumerate}
    \item 
  For all $n \le 6,$ the regular permutohedron $\Pi_n$ is BV-$\alpha$-positive. 
Therefore, all the integral generalized permutohedra (including matroid base polytopes) of dimension at most $6$ are Ehrhart positive. %Hence, all the matroid base polytopes of dimension at most $6$ are Ehrhart positive.

\item %\label{thm:34coeff}
  For any $n,$ and any face $F$ of $\Pi_n$ of codimension $2$ or $3,$ we have $\alpha(F, \Pi_n)$ is positive, where $\alpha$ is the BV-$\alpha$-valuation.

Hence, the third and fourth coefficients of the Ehrhart polynomial of any integral generalized permutohedron (including matroid base polytopes) are positive. 

\end{enumerate}

\end{thm}

%Instead of focusing on all the coefficients of Ehrhart polynomials, we also study certain special coefficients, and are able to obtain the following results.

%In addition to investigating positivity of the $\alpha$'s, there are other questions one can ask about these constructions.
During the process of investigating positivity of the $\alpha$'s, another natural question arises:
Although there are different constructions for $\alpha$, is it possible that under certain constraints, the construction is unique? Our second main result of this paper is an affirmative answer to this question when we focus on regular permutohedra. 

\begin{thm}\label{thm:uniqueness0}
	Suppose $\alpha$ is a construction such that McMullen's formula holds and it is symmetric about the coordinates (see Definition \ref{defn:symmetric}).
Then all the $\alpha$'s arising from the regular permutohedra $\Pi_n$ are uniquely determined.
\end{thm}

Even though the above theorem does not seem to be related to attempts of proving Conjecture \ref{conj:alphas},
as a consequence of techniques used in proving the theorem, we give in Corollary \ref{cor:equimixed} an equivalent statement to Conjecture \ref{conj:alphas} in terms of mixed valuations.

\subsection*{Organization of the paper}
In Section \ref{sec:background}, 
we give basic definitions and review background results that are relevant to our paper. 
In Section \ref{sec:exterior}, we give details of the BV-$\alpha$-valuation, discuss consequences of the properties of this construction. In particular, we complete the proof of Theorem \ref{thm:reduction}.

In Section \ref{sec:generic}, assuming the $\alpha$ function in McMullen's formula \eqref{equ:exterior} is symmetric about the coordinates (a property of the BV-$\alpha$-valuation), we derive a combinatorial formula for $\Lat(\Perm(\v))$ indexed by subsets of $[n],$ where $\Perm(\v)$ is a ``generic permutohedron'', which belongs to a family of generalized permutohedra containing the regular permutohedra.
We then finish with a proof for Theorem \ref{thm:partial}. %carry out direct computation to find values of BV-$\alpha$-valuations for regular permutohedra, and verify that Conjecture \ref{conj:alphas} is true for dimension up to $6.$ We are also able to verify positivity of $\alpha(F,\Pi_n)$ for faces $F$ of codimension $2$ or $3.$ There results provide evidences to Conjecture \ref{conj:alphas}. 

In Section \ref{sec:unique}, using the combinatorial formula we derived in Section \ref{sec:generic} and theories of mixed valuations, we show that the $\alpha$-values arising from the regular permutohedron are unique as long as we assume $\alpha$ is symmetric about the coordinates, as well as present equivalent statements to Conjecture \ref{conj:alphas} in terms of mixed valuations.

%In Section \ref{sec:toric}, we quickly review the connection with toric varieties as in \cite[Section 5.3]{fulton}. Analogously to the treatment of mixed volumes in \cite[Section 5.4]{fulton}, we can relate the more general mixed valuations to some intersection theoretic quantities. With these arguments we can give one more equivalence of our main Conjecture \ref{conj:alphas}. It relates to the positivity of the Todd class of the permutohedral variety (in some) expression in terms of the torus invariant cycles. 
We finish with some natural questions arising from this paper in Section \ref{sec:further}.

\subsection*{Acknowledgements}
\label{sec:ack}
The authors thank Federico Ardila and Alexander Barvinok for helpful discussions. We also thank Nicole Berline and Michele Vergne for useful comments and for providing a code used in some of the computations.

\section{Background}\label{sec:background}

In this section and the next section, we assume the ambient space is $\R^D,$ and $\Z^D$ is the \emph{lattice} in $\R^D.$ %(Whenever we discuss generalized permutohedra, $D$ will always be $n+1.$) 
For any $D$-vector $\bbeta$,  $\bbeta_i( \cdot)$ is the linear function that maps $\x \in \R^D$ to the scalar product of $\bbeta$ and $\x.$ Since we can consider $\bbeta( \cdot)$ as a point in the dual space $(\R^D)^*$ of $\R^D,$ we will use the notation $\bbeta$ (or any bold letter) to denote both the $D$-vector and the linear function.

We assume familiarity with basic definitions of polyhedra and polytopes as presented in \cite{barvinok, zie}, and only review terminologies and setups that are relevant to us. 
 A \emph{polyhedron} is the set of points defined by a linear system of %equalities and 
 inequalities. 
 %\marginpar{do we need equalities?}
 %\begin{align}
%\bbeta_i(\x) = w_i, &\quad \forall 1 \le i \le m_1,\nonumber \\
%\bgamma_i(\x) \le z_i, &\quad \forall 1 \le i \le m_2. \label{ineq}
%\end{align}
%For simplicity, we let $B$ be the $m_1 \times D$ matrix whose row vectors are $\bbeta_i$'s, $C$ the $m_2 \times D$ matrix whose row vectors are $\bgamma_i$'s, $\w = (w_1,\dots, w_{m_1})^T,$ and $\z =(z_1, \dots, z_{m_2})^T$, so the above linear system can be represented as\[ B \x = \w, \quad C \x \le \z.\] 
A \emph{polytope} is a bounded polyhedron. (A polytope can also be defined as the convex hull of a finite set of points.)
%
%For any polyhedron $P,$ we use $\vert(P)$ to denote the vertex set of $P.$
An {\it integral} polyhedron is a polyhedron whose vertices are all lattice points, i.e., points with integer coordinates.

\begin{notn}\label{notn:afflin}
For any set of points $A \subseteq \R^D,$ we denote by $\aff(A)$ the affine span of $A$, and $\lin(A)$ the linear space obtained by shifting $\aff(A)$ to the origin.
\end{notn}

Let $V$ be a subspace of $\R^D,$ and $\Lambda := V \cap \Z^D$ the lattice in $V.$ 
%For any polytope $P$ that is lying in an affine space that is a translation of $V$, 
For any polytope $P$ such that $\lin(P) \subseteq V,$ 
we define the \emph{volume of $P$ normalized to the lattice $\Lambda$} to be the integral
%\[ \vol_\Lambda(P) := \int_P \ 1 \ d  \Lambda,\]
$\ds \vol_\Lambda(P) := \int_P \ 1 \ d  \Lambda,$
where $d  \Lambda$ is the canonical Lebesgue measure defined by the lattice $\Lambda.$ %We will omit the subscript $\Lambda$ and write $\vol(P)$ if $\Lambda = \Z^n.$
In the case where $\dim P = \dim \Lambda$, we get the \emph{normalized volume} of $P$, denoted by $\nvol(P).$

\subsection{Cones and fans} A \emph{(polyhedral) cone} is the set of all nonnegative linear combinations of a finite set of vectors. %A \emph{shifted cone} is a set of points in the form of $C + \x$ where $C$ a cone and $\x$ is a point. 
A cone is \emph{pointed} if it does not contain a line.
A cone $C$ is \emph{rational} if the cone $C$ is generated by vectors with rational coordinates.

%\begin{defn}
%  Suppose $P$ is a polyhedron and $F$ is a face. The \emph{tangent cone} of $F$ at $P$ is:
%\[
%{\tcone(F,P)} = \left\{ F + \u: F + \delta \u \in P \hspace{5pt}\textrm{for sufficiently small $\delta$}\right\}.
%\]
%The \emph{feasible cone} of $F$ at $P$ is:
%\[
%{\fcone(F,P)} = \left\{ \u: F + \delta \u \in P \hspace{5pt}\textrm{for sufficiently small $\delta$}\right\}
%\]
%Note that $\tcone(F,P)$ is a shifted cone, but not necessarily a cone, when $\fcone(F,P)$ is always a cone.

%In order to always work with pointed cones, we also define
%\[{\tcone}^p(F,P) = {\tcone}(F,P)/L \quad \text{ and } \quad {\fcone}^p(F,P) = {\fcone}(F,P)/L \]
%where $L$ is the affine space spanned by $F$. Then $\tcone^p(F,P)$ and $\fcone^p(F,P)$ are pointed (shifted) cones with dimension $\dim P - \dim F$. 
%\end{defn}

\begin{defn}
  Suppose $V$ is a subspace of $\R^D$ and %$P \subset V + \y$ for some $\y \in \R^D.$
  $P$ is a polytope satisfying $\lin(P) \subseteq V.$ 
Given any face $F$ of $P$, the \emph{normal cone} of $P$ at $F$ with respect to $V$ is 
\[
\ncone_V(F, P):=\left\{ \u \in V^*:\quad \u(\p_1) \geq \u(\p_2), \quad \forall \p_1\in F,\quad \forall \p_2\in P \right\}.
\]
Therefore, $\ncone_V(F,P)$ is the collection of linear functions $\u$ in $V^*$ such that $\u$ attains maximum value at $F$ over all points in $P.$

The \emph{normal fan} $\Sigma_V(P)$ of $P$ with respect to $V$ is the collection of all normal cones of $P$. 

%In the situation where the affine span of $P$ is $V + \y,$ i.e., $\dim(P) = \dim (V),$ we will omit the subscript $V$ and the words ``with respect to $V$''.%, and just write $\ncone(F,P)$ and $\Sigma(P).$ 
\end{defn}

\begin{defn} Suppose $W$ is a subspace of $\R^D.$
	Let $K \subseteq W$ be a cone. The \emph{polar cone} of $K$ with respect to $W$ is the cone 
	\[ K_W^\circ = \{ \y \in W^* \ | \ \y(\x) \le 0,  \forall \x \in K\}.\]
In the situation where $K$ is full-dimensional in $W,$ we will omit the subscript $W$ and the words ``with respect to $W$''.%, and just write $K^\circ.$
\end{defn}
One can check that $K^\circ (= K_{{\rm span}K}^\circ)$ is always a pointed cone.
We state without proofs the following straightforward results for normal cones, which will be useful for our paper.
\begin{lem}\label{lem:ncone}
%Let $V$ and $L$ be the shifts of the affine span of $P$ and $F$, respectivly, to the origin. 
	Let $F$ be a face of a polytope $P \subseteq \R^D,$ and suppose $L = \lin(F)$ and $\lin(P) \subseteq V.$ (Recall Notation \ref{notn:afflin}.)
	Then the followings are true:
\begin{alist}
\itm $\ncone_V(F,P)$ spans the orthogonal complement of $L^*$ with respect to $V^*$. 
Hence,
%\begin{equation}\label{equ:nconedim}
$\dim \ncone_V(F, P) = \dim V - \dim F.$
%  \end{equation}
%Furthermore, the pointed feasible cone of $P$ at $F$ and the normal cone of $P$ at $F$ are polar to one another in the following sense:
%\begin{equation}\label{equ:fncone}
%  (\fcone^p(F, P))_{V/L}^\circ = \ncone_V(F, P) \text{ and }  (\ncone_V(F, P))_{V/L}^\circ = \fcone^p(F, P).
%\end{equation}
\itm The pointed cone $(\ncone_V(F,P))^\circ$ is invariant under the choice of $V$ (as long as $\lin(P) \subseteq V$). So we may omit the subscript $V$ and just write $(\ncone(F,P))^\circ$.

Furthermore, $(\ncone(F,P))^\circ$ is full-dimensional in the orthogonal complement of $L$ with respect to $\lin(P),$ and is of dimension $\dim P - \dim F.$
\end{alist}
\end{lem}
We remark that this unique pointed cone asserted in b) is known as the \emph{pointed feasible cone} of $P$ at $F,$ which is important in Berline-Vergne's construction. (See \S \ref{subsec:BV}.)

\subsection{Generalized permutohedra} 
We introduce \emph{generalized permutohedra}, the main family of polytopes we study in this paper. In this part and any later part that is related to generalized permutohedra, we assume $D = n+1,$ i.e., the ambient is $\R^{n+1}.$ First, we present the \emph{usual permutohedron} as the convex hull of a finite number of points.

\begin{defn}
  Given a point $\v = (v_1,v_2,\cdots,v_{n+1}) \in \R^{n+1}$, we define the \emph{usual permutohedron}
\[\Perm(\v) = \Perm (v_1,v_2,\cdots, v_{n+1}) := \conv\left( \left(v_{\sigma(1)},v_{\sigma(2)},\cdots, v_{\sigma({n+1})}\right):\quad \sigma\in \fS_{n+1}\right).\]
In particular, if $\v = (1, 2, \dots, {n+1}),$ we obtain the \emph{regular permutohedron}, denoted by $\Pi_{n},$
\[ \Pi_{n} := \Perm (1, 2, \dots, n+1).\]
As long as there are two different entries in $\v$ we have $\dim (\Perm(\v)) = n$.
\end{defn}

The \emph{generalized permutohedra} is originally introduced by Postnikov \cite[Definition 6.1]{bible} as polytopes obtained from usual permutohedra by moving vertices while preserving all edge directions. 
%The original definition of generalized permutohedra was given by Postnikov \cite{bible}.
In \cite{faces}, Postnikov, Reiner, and Williams give several equivalent definitions, one of which uses concepts of normal fans.

\begin{defn}
%  Let $V$ be the subspace of $\R^{n+1}$ defined by $x_1 + x_2 + \cdots + x_{n+1} = 0.$ 
  The \emph{Braid arrangement fan}, denoted by $\B_n,$  is the complete fan in $\R^{n+1}$ given by the hyperplanes
\[
  x_i - x_j = 0 \quad \text{for all $i\neq j$.}
\]
\end{defn}

\begin{prop}[Proposition 3.2 of \cite{faces}] \label{prop:coarser}
%Let $V$ be the subspace of $\R^{n+1}$ defined by $x_1 + x_2 + \cdots + \x_{n+1} = 0.$ 
  A polytope $P$ in $V:=\mathbb{R}^{n+1}$ is a generalized permutohedron if and only
if its normal fan $\Sigma_V(P)$ with respect to $V$ is refined by the Braid arrangement fan $\B_n$.
\end{prop}
It follows from \cite[Proposition 2.6]{bible} that as long as $\v=(v_1,v_2,\cdots,v_{n+1})$ has distinct coordinates, the associated usual permutohedron $\Perm(\v)$ has the Braid arrangement $\B_n$ as its normal fan. We call $\Perm(\v)$ with $\v$ of distinct coordinates a \emph{generic permutohedron}. %In particular, the regular permutohedron $\Pi_{n}$ has the Braid arrangement as its normal fan. 
In particular, the regular permutohedron $\Pi_n$ is a generic permutohedron.

\subsection{Algebra of polyhedra and mixed valuations} \label{subsec:mixed}

For a set $S \subseteq \R^D$, the \emph{indicator function} $[ S ]: \R^D
\rightarrow \R$ of $S$ is defined as $[S](x) = 1$ if $x \in S$, and $[S](x)=0$ if $x \not\in S.$
%\[[ S ] (x) = \left \{
%\begin{array}{ll} 1 \mbox{ if }x \in S, \\ 0 \mbox{ if }x \not \in
%S.\\
%\end{array}\right .\]
Let $V$ be a subspace of $\R^D.$ The \emph{algebra of polyhedra}, denoted by $\cP(V)$, is the vector space defined as the span of the indicator functions of all polyhedra in $V.$ We similarly define $\cP_b(V)$ as the \emph{algebra of polytopes}, and $\cC(V)$ as the \emph{algebra of cones}. %For any $\x \in \R^D,$ the \emph{algebra of shifted cones at $\x$}, denoted by $\cC_{\x}(V)$, is the vector space defined as the span of the indicator functions of all shifted cones that are in the form of $C +\x$ for some cone $C.$

A linear transformation $\phi: \cP(V), \cP_b(V), \cC(V) \to W,$ where $W$ is a vector space, is a \emph{valuation}.
Both volume $\vol_\Lambda( \cdot )$ and number of lattice points $\Lat( \cdot )$ are valuations on the algebra of polytopes. However, normalized volume $\nvol( \cdot )$ is not a valuation.

%%We assume the readers are familiar with the definition of algebra of polyhedra/polytopes and valuation presented in \cite{BarviPom}. 
%Below is an important result on indicator functions of feasible cones at vertices.
%\begin{thm}[Theorem 6.6 of \cite{barvinok}] \label{thm:feasible}
%Suppose $P$ is a nonempty polytope. Then
%\[
%  [0] \equiv \displaystyle \sum_{\v: \textrm{ a vertex of $P$}} [{\fcone}^p(\v,P)] \quad \textrm{modulo polyhedra with lines}
%\]
%\end{thm}

%The following lemma gives the two important equations of indicator functions of cones of polyhedra. 
%\begin{lem}[Theorem 6.4 and Problem 6.2 in \cite{barvinok}] Suppose $P$ is a non-empty polyhedron without lines. Then
%\begin{align} 
%	[ P ] \equiv& \sum_{v \in \vert(P)} [\tcone(v, P)] \ \text{ modulo polyhedra with lines;} \label{equ:tcone}\\
%	{[}  K_P ] \equiv& \sum_{v \in \vert(P)} [\fcone(v, P)] \ \text{ modulo polyhedra with lines,} \label{equ:fcone}
%\end{align}
%	where $K_P$ is the recession cone of $P.$
%
%\end{lem}

%\subsection{Mixed valuations} \label{subsec:mixed}
Let $\Lambda$ be a sublattice of $\Z^D$ and $V$ is the span of $\Lambda.$
A valuation is a \emph{$\Lambda$-valuation} if it is invariant under $\Lambda$-translation.
We say a valuation $\phi$ is \emph{homogeneous of degree $d$} if $\phi([t P]) = t^d \phi([P])$ for any integral polytope $P$ and $t \in \Z_{\ge 0}.$ It's clear that $\vol_\Lambda$ is homogeneous of degree $\dim \Lambda$, but $\Lat$ is not homogeneous. 

%We extend this terminology of homogeneous to functions taking more than one polytopes as inputs. Suppose $\cM: \underbrace{\cP_b(V) \times \cdots \times \cP_b(V)}_{d} \to W.$ We say it is \emph{homogeneous}, if 
%\[ \cM( t_1 P_1, \dots, t_d P_d) = \cM(P_1, \dots, P_d) t_1 \cdots t_d, \]
%for any integral polytopes $P_1, \dots, P_d$ and nonnegative integers $t_1, \dots, t_d.$ 

%We review the behavior of valuations with respect to Minkwoski sums. We have the following general result: 

%\begin{thm}[Theorem 6 of \cite{mcmullen}]
%\label{thm:mixed}
%Let $\phi$ be a valuation, and let $P_1,\cdots, P_k$ be integral polytopes. Then for integers $t_1,\cdots, t_k\geq 0$,
%$\phi(t_1P_1+t_2P_2+\cdots+t_kP_k)$ is a polynomial in $t_1,\cdots, t_k$.
%Moreover, the coefficient of $t_1^{r_1}\cdots t_k^{r_k}$
%is a homogeneous valuation.
%\end{thm}

The following important theorem by McMullen is a special case of \cite[Theorem 6]{mcmullen}.
\begin{thm}
	\label{thm:mixed}
Suppose $\phi$ is a homogeneous $\Lambda$-valuation on $\cP_b(V)$ of degree $d.$
Then there exists a function $\cM$ which takes $d$ integral polytopes as inputs such that
\begin{equation}\label{equ:mixed0}
  \phi(t_1P_1+t_2P_2+\cdots+t_kP_k) = \sum_{j_1,\cdots,j_d  \in [k]} \cM(P_{j_1},P_{j_2},\cdots,P_{j_d})t_{j_1}\cdots t_{j_d},
\end{equation}
for any $k \in \Z_{>0},$ any integral polytopes $P_1, \dots, P_k \subset V$ and $t_1, \dots, t_k \in \Z_{\ge 0}.$
\end{thm}

The following definition and lemma are stated in \cite[Section 3 of Chapter IV]{ewald} for the volume valuation (which is a homogeneous valuation). We give the general forms here.
\begin{defn}
  Let $\phi$ and $\cM$ be as in Theorem \ref{thm:mixed}. We define another function $\mphi$ that takes $d$ integral polytopes as inputs as an average of the function $\cM:$
\[ \mphi(P_1, \dots, P_d) := \frac{1}{d!} \sum_{\sigma \in \fS_d} \cM(P_{\sigma(1)}, \dots, P_{\sigma(d)}).\]
It is easy to see that $\mphi$ is uniquely chosen for each $\phi$, and \eqref{equ:mixed0} still holds for $\mphi:$ 
\begin{equation}\label{equ:mixed}
  \phi(t_1P_1+t_2P_2+\cdots+t_kP_k) = \sum_{j_1,\cdots,j_d \in [k]} \mphi(P_{j_1},P_{j_2},\cdots,P_{j_d})t_{j_1}\cdots t_{j_d},
\end{equation}
We call $\mphi$ the \emph{mixed valuation} of $\phi.$ 
\end{defn}

The lemma below gives two properties of the mixed valuation $\mphi.$ We omit the proof which is very similar to that is given in \cite[Section 3 of Chapter IV]{ewald} for volume valuation. 
\begin{lem}\label{lem:mixval}
  \begin{ilist}
  \itm For any integral polytopes $P_1, \dots, P_d$, and any permutation $\sigma \in \fS_d,$ we have
%\begin{equation}\label{equ:permprop}
$\mphi(P_1, \dots, P_d) = \mphi(P_{\sigma(1)}, \dots, P_{\sigma(d)}).$
%\end{equation}
  \itm The function $\mphi$ is a \emph{multi-linear} function, that is, it is linear in each component.
\end{ilist}
\end{lem}

%%%%%%%%%%% Do Not Delete the Proof %%%%%%%%%%%%%%%%
%\begin{proof} (i) follows directly from the definition of $\mphi.$ 
% For (ii), we will just prove $\mphi$ is linear in the first component, that is, to show for any integral polytopes $P_1, P_1', P_2, P_3, \dots, P_n$ and nonnegative integers $s_1, s_1',$ we have
%\begin{equation}
%  \mphi(s_1 P_1 + s_1' P_1', P_2, \dots, P_d) = s_1 \mphi(P_1, P_2, \dots, P_d) + s_1' \mphi(P_1', P_2, \dots, P_d), \label{equ:multilinear}
%\end{equation}
%
%We apply \eqref{equ:mixed} to both sides of the following equality:
%\[\phi( t_1 (s_1 P_1+ s_1' P_1') + t_2 P_2 + \cdots + t_d P_d) = \phi( (t_1 s_1) P_1 + (t_1 s_1') P_1' + t_2 P_2 + \cdots + t_d P_d).\]
%Consider $s_1$ and $s_1'$ as fixed numbers. Then each side gives a homogeneous polynomial in $t_1, t_2, \dots, t_d.$ Since these two homogeneous polynomials agree on all $t_1, t_2, \dots, t_d \in \Z_{\ge 0},$ we conclude that they are exactly the same polynomials, and thus their coefficients agree. Then \eqref{equ:multilinear} follows from \eqref{equ:permprop} and comparing the coefficients of $t_1 t_2 \dots t_d.$ 
%\end{proof}
%%%%%%%%%%%%%%%%%%%%%%%%%%%%%%%%%%%%%%%%%%%

Apply the above results to volume valuation, a homogeneous valuation, we obtain the following:
\begin{thm}[Theorem 3.2 of \cite{ewald}] \label{thm:mixedvol}
  Suppose $P_1,\cdots, P_k$ are integral polytopes with $\dim(P_1+\cdots +P_k)=d$. Let $\Lambda$ be the $d$-dimensional lattice $\mathrm{span}(P_1+\cdots + P_k) \cap \Z^D.$ Then
\[
  \vol_\Lambda(t_1P_1+t_2P_2+\cdots+t_kP_k) = \sum_{j_1,\cdots,j_d \in [k]} \mv_\Lambda(P_{j_1},P_{j_2},\cdots,P_{j_d})t_{j_1}\cdots t_{j_d}
\]
where the sum is carried out independently over the $j_i$. The function $\mv_\Lambda(P_{j_1},P_{j_2},\cdots,P_{j_d})$ is called the \emph{mixed volume} of $P_{j_1},P_{j_2},\cdots,P_{j_d}$.
\end{thm}
Furthermore we have the following properties:
\begin{thm}[Theorem 4.13 of \cite{ewald}] \label{thm:mixvolprop}
Let $P_1,\cdots,P_d$ be integral polytopes. Then,
\begin{enumerate}
\item $\mv_\Lambda(P_1,\cdots, P_d) \geq 0 $
\item $\mv_\Lambda(P_1,\cdots, P_d) > 0$ if and only if each $P_i$ contains a line segment $I_i=[a_i,b_i]$ such that $b_1-a_1,\cdots,
b_d-a_d$ are linearly independent.
\end{enumerate}
\end{thm}

The lattice point, or counting, valuation $\Lat$ is not homogeneous. However it can be decomposed into homogeneous parts.
\begin{thm}[Theorem 5 of \cite{mcmullen}] \label{thm:decomposeLat}
Suppose $V$ is $d$-dimensional. Then we can decompose the valuation $\Lat$ as
\[
\Lat = \Lat^d + \cdots +\Lat^1 + \Lat^0
\]
where $\Lat^r$ is homogeneous of degree $r$.
\end{thm}
This decomposition corresponds to the coefficients of the Ehrhart polynomial, in particular $\Lat^d$ corresponds to the volume valuation $\vol_\Lambda,$ where $\Lambda = V \cap \Z^D$.
Applying Theorem \ref{thm:mixed} and Lemma \ref{lem:mixval} to each homogeneous function $\Lat^r$ gives us the following result:%the result on the \emph{mixed lattice point enumerator}:

\begin{thm}\label{thm:mixedLat}
Suppose $P_1,\cdots, P_k$ are integral polytopes with $\dim(P_1+\cdots +P_k)=d$. Then \[
\Lat(t_1P_1+t_2P_2+\cdots+t_kP_k) = \sum_{e=0}^d \sum_{j_1,\cdots,j_e =1}^k \ml^e(P_{j_1},P_{j_2},\cdots,P_{j_e})t_{j_1}\cdots t_{j_e}.
\]
%where $\ml^e(P_{j_1}, P_{j_2}, \dots, P_{j_e})$ is called the \emph{mixed lattice point enmerator} of $P_{j_1}, P_{j_2}, \dots, P_{j_e}.$
\end{thm}

We cannot expect $\ml^r$ or any other mixed valuation $\mphi$ to be nonnegative in general. However we have a way to compute them. 

\begin{thm}\label{thm:mobius}
Suppose $\phi$ is a homogeneous $\Lambda$-valuation on $\cP_b(V)$ of degree $d.$ For any integral polytopes $P_1,P_2,\cdots, P_d \subset V,$ we have
\[
d! \mphi(P_1,P_2,\cdots, P_d) = \sum_{J\subseteq [d]} \left(-1\right)^{d-|J|} \phi\left(\sum_{j\in J}P_j\right)
\]
\end{thm}
\begin{proof}
  We define two functions $f$ and $g$ on the Boolean algebra of order $d$. (See \cite[Chapter 3]{enum}.) For any subset $T \subseteq [d],$ let 
\[
  f(T) := \phi\left(\sum_{i\in T}P_i\right) \quad \text{and} \quad g(T) := \sum_{\substack{j_1, \dots, j_d \in [d] \\ \cup_{i=1}^d \{j_i\} = T}} \mphi \left(P_{j_1}, P_{j_2},\cdots, P_{j_d}\right).
\]
Apply \eqref{equ:mixed} with $t_i=1$ to $f(T),$ one sees that $f(T) = \sum_{S \subseteq T} g(S).$ Therefore, by Mobius inversion, we get
\[ g(T) = \sum_{S \subseteq T} (-1)^{|T|-|S|} f(S).\]
Then the theorem follows from evaluating the above equality at $T=[d].$
\end{proof}

\section{McMullen's formula and the BV-$\alpha$-valuation}\label{sec:exterior}
Recall that in the introduction we discuss the existence of the following \emph{McMullen's formula}: 
\begin{equation}\label{equ:exterior1}
\Lat(P) = \displaystyle \sum_{F: \textrm{ a face of $P$}} \alpha(F,P)
\ \nvol(F), \end{equation}
where $\alpha(F,P)$ depends only on the normal cone of $P$ at $F$.

One immediate consequence of the existence of McMullen's formula \eqref{equ:exterior1} is that it provides another way to prove Ehrhart's theorem. Moreover, it gives a description of each Ehrhart coefficient. We state the following modified version of Theorem \ref{thm:ehrhart0}.

\begin{thm}\label{thm:ehrhart} 
  For an integral polytope $P \subset \Z^D$ and any $t \in \Z_{\ge 0}$, the function
%\[i(P, t) =  \Lat(t P) = | tP \cap \Z^n| \]
$i(P, t) =  \Lat(t P) = | tP \cap \Z^n|$
is a polynomial in $t$ of degree $\dim P.$ Furthermore, the coefficient of $t^k$ in $i(P,t)$ is given by
\begin{equation}\label{equ:coeff}
  \sum_{F: \text{ a $k$-dim face of $P$}} \alpha(F, P) \ \nvol(F).
\end{equation}
\end{thm}

\begin{proof}
The desired formula follows from applying McMullen's formula to $tP$ and observing that $\alpha(tF, tP) =\alpha(F,P).$
%When we dilate the polytope $P$ by a factor of $t$, each face $F$ of $P$ becomes $tF,$ a face of $tP.$ It is clear that the normal cone does not change. Hence, applying McMullen's formula to $tP,$ we get
%\[i(P,t) = \sum_F \alpha(tF,tP) \ \nvol(tF) = \sum_F \alpha(F,P) \ \nvol(F) \ t^{\dim F}.\] 
%Then our conclusion follows.
\end{proof}

Formula \eqref{equ:coeff} for the coefficients of the Ehrhart polynomial $i(P,t)$ gives a sufficient condition for Ehrhart positivity. %It is our first step to reduce conjecture \ref{positivity} to conjecture \ref{conj:alphas}. 
\begin{lem}\label{lem:red1}
Let $P$ be an integral polytope. For a fixed $k,$ if $\alpha(F, P)$ is positive for any $k$-dimensional face of $P,$ then the coefficient of $t^k$ in the Ehrhart polynomial $i(P,t)$ of $P$ is positive.

Hence, if $\alpha(F,P)>0$ for every face $F$ of $P$, then $P$ is Ehrhart positive.
\end{lem}

As discussed in the introduction, different constructions of $\alpha(F, P)$ were given in the literature. In our paper, we will use Berline-Vergne's construction, which we refer to as the \emph{BV-$\alpha$-valuation}.

\subsection{Berline-Vergne's construction} \label{subsec:BV}
%Berline and Vergne construct in \cite{localformula} a function $\Psi([C], \Lambda)$ on indicator functions of rational shifted cones $C$ with respect to a lattice $\Lambda,$ where $\Lambda$ is a quotient of the lattice $V \cap \Z^D$ and $C$ is inside the affine span of $\Lambda.$
%Then they show $\Psi$ has the following properties:
For any rational quotient subspace $W$ of $\R^D,$ Berline and Vergne construct in \cite{localformula} a function $\Psi_W([C])$ on indicator functions of rational cones $C$ in $W$ with the following properties:

%\begin{thm}
%There exist a valuation $\alpha$ on rational cones with the following properties:
\begin{enumerate}[(P1)]
  %\item Let $V$ be the affine span of a lattice $\Lambda.$ Then $\Psi( \cdot, \Lambda)$ is a valuation on $\cC_\x(V)$ for any $\x \in V.$ (Recall that $\cC_\x(V)$ is the algebra of shifted cones at $\x.$)
  \item $\Psi_W( \cdot)$ is a valuation on the algebra of rational cones in $W$. %$\cC( \R^D)$. (Recall that $\cC_\x(V)$ is the algebra of cones.)
  \item McMullen's formula \eqref{equ:exterior1} holds for integral polytopes in $\R^D$ if we set 
  \begin{equation}\label{equ:defnalpha}
	  \alpha(F,P) := \Psi_{\R^D/\lin(F)}([ (\ncone(F,P))^\circ ]).
\end{equation}
\item If a cone $C$ contains a line, then $\Psi_W([C])=0$.
\item Its value on the zero-dimensional cone $\0$ is 1, i.e. $\Psi_W([\0])=1$.
\item $\Psi_W$ is invariant under the action of $O_D(\Z)$, the group of orthogonal unimodular transformations. More precisely, if $T$ is an orthogonal unimodular transformation, for any cone $C,$ we have $\Psi_W([C]) = \Psi_{T(W)}([T(C)]).$
 
\item Let $L$ be the orthogonal complement of $\lin(C)$. Then $\Psi_W([C]) = \Psi_{\R^D/L}([C]).$ (This indicates that we may omit the subscript $W$ in $\Psi_W.$)
\end{enumerate}
%\end{thm}

%For the rest of the section, we assume $\alpha(F,P)$ comes from the BV-$\alpha$-valuation. 

It is important to remark that their main construction is actually a valuation with values on certain classes of holomorphic functions, and the function $\Psi_W$ we described above comes from the (non-trivial) evaluation at zero. For more details we refer the reader to the original paper \cite{bvtodd} and to the exposition in \cite[Chapters 19-20]{barvinok}.

\subsection{Reduction theorem}
We've already discussed a consequence of the existence of McMullen's formula, which reduces the problem of proving Ehrhart positivity to proving $\alpha$-positivity. Now we will discuss a very important consequence of the BV-$\alpha$-valuation -- the reduction theorem -- applying which we complete the proof of Theorem \ref{thm:reduction}. %-- which will be applied to generalized permutohedra in Section \ref{sec:general}. 

%For the rest of the section, we assume $\alpha(F,P)$ comes from the BV-$\alpha$-valuation. Also, because we only deal with integral polytopes, we will take \eqref{equ:defnalpha1} as the definition of $\alpha(F,P).$
For the rest of the section, we assume $\alpha(F,P)$ comes from the BV-$\alpha$-valuation, and take \eqref{equ:defnalpha} as the definition of $\alpha(F,P).$
We start with the following preliminary lemma.
\begin{lem}
 %Suppose $V$ is a subspace of $\R^D,$ and $P \subset V + \x$ and $Q \subset V + \y$ are two integral polytopes for some points $\x$ and $\y.$ 
Suppose $V$ is a subspace of $\R^D$, and $P$ and $Q$ are two integral polytopes in $V.$
  Let $F$ be a face of $P.$ Suppose there exist faces $G_1, G_2, \dots, G_r$ of $Q$ of the same dimension such that
  \begin{equation}\label{equ:unionncone}
    \ncone_V(F, P) = \cup_{i=1}^r \ncone_V(G_i, Q).
\end{equation}
Then $F$ is of the same dimension as $G_i$'s, and 
$\ds \alpha(F, P) = \sum_{i=1}^r \alpha(G_i, Q).$
\end{lem}

\begin{proof}
  The first consequence of Equation \eqref{equ:unionncone} is that $\ncone_V(F,P)$ and $\ncone_V(G_i, Q)$'s all span the same subspace, say $W^*$. Let $L$ be the orthogonal complement of $W$ with respect to $V.$ Then by Lemma \ref{lem:ncone}/a), %affine spans of $F$ and $G_i$'s are all shifts of $L.$ 
  \[\lin(F) = L = \lin(G_i) \quad \text{for each $i.$}\]
  Hence, $F$ has the same dimension as $G_i$'s. %Letting $\Lambda := (V \cap \Z^D)/L$, we have
% \[ \alpha(F,P) = \Psi(\fcone^p(F,P), \Lambda), \quad \text{and} \quad \alpha(G_i,Q) = \Psi(\fcone^p(G_i,Q), \Lambda) \ \forall i.\] 
%
Next, since $\ncone_V(G_i, Q) \cap \ncone_V(G_j, Q)$ is a lower dimensional cone for any $i \neq j,$ we have 
  \[ [ \ncone_V(F, P) ] \equiv \sum_{i=1}^r [ \ncone_V(G_i, Q)] \quad \text{modulo polyhedra contained in proper subspaces of $W$}.\]
  Taking the polar of the above identity with respect to $W$ and applying Lemma \ref{lem:ncone}/b) yields %\eqref{equ:fncone} yields
  \[ [ \ncone(F, P) ]^\circ \equiv \sum_{i=1}^r [ \ncone(G_i, Q)]^\circ \quad \text{modulo polyhedra with lines}.\]
% \[\left[ \fcone^p(F, P)\right] \equiv \sum_{i=1}^r \left[\fcone^p(G_i,Q)\right]\quad\text{modulo polyhedra with lines}.\]
%  Applying Corollary \ref{cor:sum}, we obtained the desired identity.
  Then our desired identity follows from Properties (P1) and (P3) of the Berline-Vergne construction.
\end{proof}
Our reduction theorem is a clear consequence of the above lemma.
\begin{thm}[Reduction Theorem] \label{thm:reduction-gen}
Suppose $V$ is a subspace of $\R^D,$ and %$P \subset V + \x$ and $Q \subset V + \y$ are two integral polytopes for some points $\x$ and $\y.$ 
$\lin(P)$ and $\lin(Q)$ are both subspaces of $V.$
%Let $P$ and $Q$ be two integral polytopes in $V$, a subspace of $\R^n.$ 
Assume further the normal fan $\Sigma_V(P)$ of $P$ with respect to $V$ is a refinement of the normal fan $\Sigma_V(Q)$ of $Q$ with respect to $V$. 
Then for any fixed $k,$ if $\alpha(F, P) >0 $ for every $k$-dimensional face $F$ of $P$, then $\alpha(G, Q) >0$ for every $k$-dimensional face $G$ of $Q.$

Therefore, BV-$\alpha$-positivity of $P$ implies BV-$\alpha$-positivity of $Q.$
\end{thm}

Theorem \ref{thm:reduction-gen} and Proposition \ref{prop:coarser} immediately give the following result and complete the proof for Theorem \ref{thm:reduction}
\begin{thm}[Reduction Theorem, special form]\label{thm:reduction-spe}
Let $Q \subset \R^{n+1}$ be a generalized permutohedron. 
Then for any fixed $k,$ if $\alpha(F, \Pi_{n}) >0 $ for every $k$-dimensional face $F$ of $\Pi_{n}$, then $\alpha(G, Q) >0$ for every $k$-dimensional face $G$ of $Q.$

Therefore, BV-$\alpha$-positivity of $\Pi_{n}$ implies BV-$\alpha$-positivity of $Q.$\end{thm}

\begin{proof}[Proof of Theorem \ref{thm:reduction}]
 The theorem follows from Theorem \ref{thm:reduction-spe} and Lemma \ref{lem:red1}.
\end{proof}

\begin{rem}\label{rem:upton}
  All permutohedra of dimension at most $n$ can be realized in $\R^{n+1}$, and Proposition \ref{prop:coarser} applies to all of these permutohedra. Therefore, the polytope $Q$ in Theorem \ref{thm:reduction-spe} could be any permutohedron of dimension up to $n.$
\end{rem}

\begin{rem}\label{rem:generic}
  %Let $\v$ be a vector with distinct coordinates. Then
	Theorem \ref{thm:reduction-spe} still holds if we replace $\Pi_{n}$ with any generic permutohedron, that is, any $\Perm(\v)$ where $\v \in \R^{n+1}$ is a vector with distinct coordinates.
  %$\Perm(\v)$.
\end{rem}

\subsection{Examples of computing $\Psi_W$} \label{subsec:excmp}
%Below we will give examples of computing $\Psi$ and using which to get $\alpha(F,P)$. 
Let $W =\R^D/L$ be a rational quotient space of $\R^D.$ We might identify $W$ with $L^\perp$, the orthogonal complement of $L$ with respect to $\R^D.$ We consider 
\[ \Lambda(W) := \text{ the orthogonal projection of $\Z^D$ to $L^\perp$} \]
	to be the lattice of $W,$ which is important in computing $\Psi_W$ associated to Berline-Vergne's construction.

The computation of the function $\Psi_W$ %associated to Berline-Vergne's construction 
is carried out recursively. Hence, it is quicker to compute $\Psi_W$ for lower dimensional cones. 
Since the dimension of $(\ncone(F, P))^\circ$ is equal to the codimension of $F$ with respect to $P$, the value of $\alpha(F,P)$ is easier to compute if $F$ is a higher dimensional face. The following easy results follow from comments in \cite[Example 20.2]{barvinok}.
\begin{lem}\label{lem:01dim} Suppose $\alpha$ is the BV-$\alpha$-valuation. Then for any integral polytope $P$ and any facet $F$ of $P,$ we have
$\alpha(P,P) = 1 \text{ and } \alpha(F, P) = 1/2.$
\end{lem}

In general, the computation of $\Psi_W([C])$ is quite complicated. However, when $C$ is a \emph{unimodular cone} with respect to the lattice $\Lambda(W)$, that is, $C$ is generated by a set of rays that can be extended to a basis of $\Lambda(W)$, computations are greatly simplified. 
In small dimensions we can even give a simple closed expression for $\Psi_W$ of unimodular cones. The following result is given in \cite[Example 19.3]{barvinok}.

\begin{lem}\label{lem:2dim} %Suppose $C \subseteq \R^D$ and $\Lambda$ is the orthogonal projection of the lattice $\Z^D$ to the subspace ${\rm span}(C).$
  %If $C = \cone(\u_1, \u_2),$ where $\{\u_1, \u_2\}$ is a basis of the lattice $\Lambda$, then 
  Suppose $C = \cone(\u_1, \u_2) \subset W,$ where $\{\u_1, \u_2\}$ can be extended to a basis of the lattice $\Lambda(W)$. Then 
  \[\Psi_W([C])=\displaystyle \frac{1}{4}+\frac{1}{12}\left(\frac{\langle \u_1,\u_2\rangle}{\langle \u_1,\u_1\rangle}+\frac{\langle \u_1, \u_2\rangle}{\langle \u_2,\u_2\rangle}\right).\]
\end{lem}

%We finish this part with a formula for computing $\Psi$ of a $3$-dim unimodular cone, which was computed from Maple code.

With the help of Maple code provided by Vergne, one can obtain a formula for computing $\Psi_W$ of a $3$-dimensional unimodular cone.
\begin{lem}\label{lem:3dim} %Suppose $C \subseteq \R^D$ and $\Lambda$ is the orthogonal projection of the lattice $\Z^D$ to the subspace ${\rm span}(C).$
%If $C = \cone(\u_1,\u_2,\u_3)$ where $\u_1, \u_2, \u_3$ is a basis of the lattice $\Lambda,$ then

  Suppose $C = \cone(\u_1, \u_2, \u_3) \subset W,$ where $\{\u_1, \u_2, \u_3\}$ can be extended to a basis of the lattice $\Lambda(W)$. Then 
\[
  \Psi_W([C])= \displaystyle \frac{1}{8}+\frac{1}{24}\left( \frac{\langle \u_1, \u_2\rangle}{\langle \u_1, \u_1\rangle}+\frac{\langle \u_1, \u_2\rangle}{\langle \u_2,\u_2\rangle}+ \frac{\langle \u_1, \u_3\rangle}{\langle \u_1,\u_1\rangle}+ \frac{\langle \u_1, \u_3\rangle}{\langle \u_3, \u_3\rangle}+\frac{\langle \u_3, \u_2\rangle}{\langle \u_2, \u_2\rangle}+ \frac{\langle \u_3, \u_2\rangle}{\langle \u_3,\u_3\rangle} \right). 
\]
\end{lem}

\begin{rem}
The formulas for $2$-dimensional and $3$-dimensional unimodular cones appear to be simple. However, the apparent simplicity breaks down for dimension $4$. The formula for $4$-dimensional unimodular cones include (way) more than $1000$ terms.
\end{rem}

\subsection{Symmetry Property}\label{subsec:moreappl} 
Theorem \ref{thm:reduction-gen} mainly follows from Properties (P1) and (P3) of the BV-$\alpha$-valuation. In this subsection, we will focus on Property (P5) (as well as Property (P6)), showing reasons why this particular construction of $\alpha$ is nice.

%We remark that Lemma \ref{lem:sum} also directly gives the results on the three coefficients of $i(P,t)$ with explicit simple descriptions: 
%the leading coefficient is equal to the normalized volume of $P$, the second coefficient is one half of the sum of the normalized volumes of facets, and the constant term is always $1$. This pattern extend naturally to \emph{boxes}.

\begin{lem}\label{lem:symmetric0}
  The valuation $\Psi_W$ is \emph{symmetric about the coordinates}, i.e., for any cone $C \in W$ and any permutation $\sigma \in \fS_D,$ we have 
  \[ \Psi_W([C]) = \Psi_{\sigma(W)}([\sigma(C)]),\]
  where $\sigma(T) = \{ (x_{\sigma(1)}, x_{\sigma(2)}, \dots, x_{\sigma(D)}) \ : \ (x_1,\dots, x_D)\in T\}$ for any set $T \subseteq \R^D.$
\end{lem}

\begin{proof}
Let $M_\sigma$ be the permutation matrix corresponding to $\sigma.$ Then the lemma follows from the observation that $T$ is mapped to $\sigma(T)$ under the linear transformation $M_\sigma$ and any permutation matrix is orthonormal and unimodular.
\end{proof}

The above result motivates the following definition.
\begin{defn}\label{defn:symmetric}
  Suppose $\alpha$ is a construction such that McMullen's formula \eqref{equ:exterior1} holds. We say it is \emph{symmetric about the coordinates} 
  if $\alpha(F, P) = \alpha(G, Q)$ 
  whenever $ (\ncone(F,P))^\circ = \sigma \left( (\ncone(G,Q))^\circ\right)$ for some $\sigma \in \fS_D.$
\end{defn}

Therefore, we have the following:
\begin{lem}\label{lem:symmetric}
The BV-$\alpha$-valuation is symmetric about the coordinates.
\end{lem}

\begin{proof}
  Suppose $\Psi_W$ is the construction given by Berline-Vergne described in \S \ref{subsec:BV}.  
  Let $C = (\ncone(F,P))^\circ$ and $C' = (\ncone(G,Q))^\circ.$ Then by Property (P5), one sees that it is enough to show that \[\Psi_{\R^D/\lin(F)}([C]) = \Psi_{\R^D/\sigma(\lin(G))}([C]).\]
which follows from Property (P6). 
\end{proof}

%
%From results we show above, one sees that Property (P6) is an important property for the BV-$\alpha$-valuation. It provides us a way to obtain $\alpha$-values for special situations without explicit computation for $\Psi$ which could be quite complicated. (See examples in the next subsection.) In fact, the following lemma, which states a special case of this property, will be applied extensively when we compute $\alpha$-vaules for regular permutohedra in Section \ref{sec:generic}.

Below we give an example of how to use Lemma \ref{lem:symmetric} to compute BV-$\alpha$-values directly without using Berline-Vergne's recursive computation for $\Psi_W$ (which could be quite complicated as we've seen in \S \ref{subsec:excmp}). 
\begin{ex}\label{ex:3d}
%Here we show an example of how to use the properties of BV-$\alpha$'s to find its values. Let's consider the permutohedron in dimension 3. 
	Consider the regular permutohedron $\Pi_3$. (See the polytope on the left below.)
\begin{figure}[h]
\begin{tikzpicture}%
	\begin{scope}	[x={(0.126276cm, 0.984539cm)},
	y={(-0.229499cm, -0.090066cm)},
	z={(0.965083cm, -0.150240cm)},
	scale=1.000000,
	back/.style={loosely dotted, thin},
	edge/.style={color=blue!95!black, thick},
	facetred/.style={fill=red!95!black,fill opacity=0.200000},
	facet/.style={fill=blue!95!black,fill opacity=0.200000},
	vertex/.style={inner sep=1pt,circle,draw=green!25!black,fill=green!75!black,thick,anchor=base}]
%
%
%% Coordinate of the vertices:
%%
\coordinate (-2.00000, -1.00000, 0.00000) at (-2.00000, -1.00000, 0.00000);
\coordinate (-2.00000, 0.00000, -1.00000) at (-2.00000, 0.00000, -1.00000);
\coordinate (-2.00000, 0.00000, 1.00000) at (-2.00000, 0.00000, 1.00000);
\coordinate (-2.00000, 1.00000, 0.00000) at (-2.00000, 1.00000, 0.00000);
\coordinate (-1.00000, -2.00000, 0.00000) at (-1.00000, -2.00000, 0.00000);
\coordinate (-1.00000, 0.00000, -2.00000) at (-1.00000, 0.00000, -2.00000);
\coordinate (-1.00000, 0.00000, 2.00000) at (-1.00000, 0.00000, 2.00000);
\coordinate (-1.00000, 2.00000, 0.00000) at (-1.00000, 2.00000, 0.00000);
\coordinate (0.00000, -2.00000, -1.00000) at (0.00000, -2.00000, -1.00000);
\coordinate (0.00000, -2.00000, 1.00000) at (0.00000, -2.00000, 1.00000);
\coordinate (0.00000, -1.00000, -2.00000) at (0.00000, -1.00000, -2.00000);
\coordinate (0.00000, -1.00000, 2.00000) at (0.00000, -1.00000, 2.00000);
\coordinate (0.00000, 1.00000, -2.00000) at (0.00000, 1.00000, -2.00000);
\coordinate (0.00000, 1.00000, 2.00000) at (0.00000, 1.00000, 2.00000);
\coordinate (0.00000, 2.00000, -1.00000) at (0.00000, 2.00000, -1.00000);
\coordinate (0.00000, 2.00000, 1.00000) at (0.00000, 2.00000, 1.00000);
\coordinate (1.00000, -2.00000, 0.00000) at (1.00000, -2.00000, 0.00000);
\coordinate (1.00000, 0.00000, -2.00000) at (1.00000, 0.00000, -2.00000);
\coordinate (1.00000, 0.00000, 2.00000) at (1.00000, 0.00000, 2.00000);
\coordinate (1.00000, 2.00000, 0.00000) at (1.00000, 2.00000, 0.00000);
\coordinate (2.00000, -1.00000, 0.00000) at (2.00000, -1.00000, 0.00000);
\coordinate (2.00000, 0.00000, -1.00000) at (2.00000, 0.00000, -1.00000);
\coordinate (2.00000, 0.00000, 1.00000) at (2.00000, 0.00000, 1.00000);
\coordinate (2.00000, 1.00000, 0.00000) at (2.00000, 1.00000, 0.00000);
%%
%%
%% Drawing edges in the back
%%
\draw[edge,back] (-2.00000, -1.00000, 0.00000) -- (-2.00000, 0.00000, -1.00000);
\draw[edge,back] (-2.00000, -1.00000, 0.00000) -- (-2.00000, 0.00000, 1.00000);
\draw[edge,back] (-2.00000, -1.00000, 0.00000) -- (-1.00000, -2.00000, 0.00000);
\draw[edge,back] (-1.00000, -2.00000, 0.00000) -- (0.00000, -2.00000, -1.00000);
\draw[edge,back] (-1.00000, -2.00000, 0.00000) -- (0.00000, -2.00000, 1.00000);
\draw[edge,back] (-1.00000, 0.00000, -2.00000) -- (0.00000, -1.00000, -2.00000);
\draw[edge,back] (0.00000, -2.00000, -1.00000) -- (0.00000, -1.00000, -2.00000);
\draw[edge,back] (0.00000, -2.00000, -1.00000) -- (1.00000, -2.00000, 0.00000);
\draw[edge,back] (0.00000, -2.00000, 1.00000) -- (0.00000, -1.00000, 2.00000);
\draw[edge,back] (0.00000, -2.00000, 1.00000) -- (1.00000, -2.00000, 0.00000);
\draw[edge,back] (0.00000, -1.00000, -2.00000) -- (1.00000, 0.00000, -2.00000);
\draw[edge,back] (1.00000, -2.00000, 0.00000) -- (2.00000, -1.00000, 0.00000);

\fill[facetred] (-2.00000, -1.00000, 0.00000) -- (-2.00000, 0.00000, 1.00000) -- (-2.00000, 1.00000, 0.00000) -- (-2.00000, 0.00000, -1.00000) -- cycle {};
\fill[facetred] (-1.00000, 0.00000, -2.00000) -- (0.00000, 1.00000, -2.00000) -- (1.00000, 0.00000, -2.00000) -- (0.00000, -1.00000, -2.00000) -- cycle {};
\fill[facetred] (-1.00000, -2.00000, 0.00000) -- (0.00000, -2.00000, 1.00000) -- (1.00000,-2.00000, 0.00000) -- (0.00000, -2.00000, -1.00000) -- cycle {};

%% Drawing vertices in the back
%%
\node[vertex] at (-2.00000, -1.00000, 0.00000)     {};
\node[vertex] at (-1.00000, -2.00000, 0.00000)     {};
\node[vertex] at (0.00000, -2.00000, -1.00000)     {};
\node[vertex] at (0.00000, -1.00000, -2.00000)     {};
\node[vertex] at (0.00000, -2.00000, 1.00000)     {};
\node[vertex] at (1.00000, -2.00000, 0.00000)     {};
%%
%%
%% Drawing the facets
%%
\fill[facet] (2.00000, 1.00000, 0.00000) -- (1.00000, 2.00000, 0.00000) -- (0.00000, 2.00000, 1.00000) -- (0.00000, 1.00000, 2.00000) -- (1.00000, 0.00000, 2.00000) -- (2.00000, 0.00000, 1.00000) -- cycle {};
\fill[facet] (2.00000, 1.00000, 0.00000) -- (1.00000, 2.00000, 0.00000) -- (0.00000, 2.00000, -1.00000) -- (0.00000, 1.00000, -2.00000) -- (1.00000, 0.00000, -2.00000) -- (2.00000, 0.00000, -1.00000) -- cycle {};
\fill[facetred] (2.00000, 1.00000, 0.00000) -- (2.00000, 0.00000, -1.00000) -- (2.00000, -1.00000, 0.00000) -- (2.00000, 0.00000, 1.00000) -- cycle {};
\fill[facetred] (1.00000, 0.00000, 2.00000) -- (0.00000, -1.00000, 2.00000) -- (-1.00000, 0.00000, 2.00000) -- (0.00000, 1.00000, 2.00000) -- cycle {};
\fill[facet] (0.00000, 2.00000, -1.00000) -- (-1.00000, 2.00000, 0.00000) -- (-2.00000, 1.00000, 0.00000) -- (-2.00000, 0.00000, -1.00000) -- (-1.00000, 0.00000, -2.00000) -- (0.00000, 1.00000, -2.00000) -- cycle {};
\fill[facet] (0.00000, 2.00000, 1.00000) -- (-1.00000, 2.00000, 0.00000) -- (-2.00000, 1.00000, 0.00000) -- (-2.00000, 0.00000, 1.00000) -- (-1.00000, 0.00000, 2.00000) -- (0.00000, 1.00000, 2.00000) -- cycle {};
\fill[facetred] (1.00000, 2.00000, 0.00000) -- (0.00000, 2.00000, -1.00000) -- (-1.00000, 2.00000, 0.00000) -- (0.00000, 2.00000, 1.00000) -- cycle {};
%%
%%
%% Drawing edges in the front
%%
\draw[edge] (-2.00000, 0.00000, -1.00000) -- (-2.00000, 1.00000, 0.00000);
\draw[edge] (-2.00000, 0.00000, -1.00000) -- (-1.00000, 0.00000, -2.00000);
\draw[edge] (-2.00000, 0.00000, 1.00000) -- (-2.00000, 1.00000, 0.00000);
\draw[edge] (-2.00000, 0.00000, 1.00000) -- (-1.00000, 0.00000, 2.00000);
\draw[edge] (-2.00000, 1.00000, 0.00000) -- (-1.00000, 2.00000, 0.00000);
\draw[edge] (-1.00000, 0.00000, -2.00000) -- (0.00000, 1.00000, -2.00000);
\draw[edge] (-1.00000, 0.00000, 2.00000) -- (0.00000, -1.00000, 2.00000);
\draw[edge] (-1.00000, 0.00000, 2.00000) -- (0.00000, 1.00000, 2.00000);
\draw[edge] (-1.00000, 2.00000, 0.00000) -- (0.00000, 2.00000, -1.00000);
\draw[edge] (-1.00000, 2.00000, 0.00000) -- (0.00000, 2.00000, 1.00000);
\draw[edge] (0.00000, -1.00000, 2.00000) -- (1.00000, 0.00000, 2.00000);
\draw[edge] (0.00000, 1.00000, -2.00000) -- (0.00000, 2.00000, -1.00000);
\draw[edge] (0.00000, 1.00000, -2.00000) -- (1.00000, 0.00000, -2.00000);
\draw[edge] (0.00000, 1.00000, 2.00000) -- (0.00000, 2.00000, 1.00000);
\draw[edge] (0.00000, 1.00000, 2.00000) -- (1.00000, 0.00000, 2.00000);
\draw[edge] (0.00000, 2.00000, -1.00000) -- (1.00000, 2.00000, 0.00000);
\draw[edge] (0.00000, 2.00000, 1.00000) -- (1.00000, 2.00000, 0.00000);
\draw[edge] (1.00000, 0.00000, -2.00000) -- (2.00000, 0.00000, -1.00000);
\draw[edge] (1.00000, 0.00000, 2.00000) -- (2.00000, 0.00000, 1.00000);
\draw[edge] (1.00000, 2.00000, 0.00000) -- (2.00000, 1.00000, 0.00000);
\draw[edge] (2.00000, -1.00000, 0.00000) -- (2.00000, 0.00000, -1.00000);
\draw[edge] (2.00000, -1.00000, 0.00000) -- (2.00000, 0.00000, 1.00000);
\draw[edge] (2.00000, 0.00000, -1.00000) -- (2.00000, 1.00000, 0.00000);
\draw[edge] (2.00000, 0.00000, 1.00000) -- (2.00000, 1.00000, 0.00000);
%%
%%
%% Drawing the vertices in the front
%%
\node[vertex] at (-2.00000, 0.00000, -1.00000)     {};
\node[vertex] at (-2.00000, 0.00000, 1.00000)     {};
\node[vertex] at (-2.00000, 1.00000, 0.00000)     {};
\node[vertex] at (-1.00000, 0.00000, -2.00000)     {};
\node[vertex] at (-1.00000, 0.00000, 2.00000)     {};
\node[vertex] at (-1.00000, 2.00000, 0.00000)     {};
\node[vertex] at (0.00000, -1.00000, 2.00000)     {};
\node[vertex] at (0.00000, 1.00000, -2.00000)     {};
\node[vertex] at (0.00000, 1.00000, 2.00000)     {};
\node[vertex] at (0.00000, 2.00000, -1.00000)     {};
\node[vertex] at (0.00000, 2.00000, 1.00000)     {};
\node[vertex] at (1.00000, 0.00000, -2.00000)     {};
\node[vertex] at (1.00000, 0.00000, 2.00000)     {};
\node[vertex] at (1.00000, 2.00000, 0.00000)     {};
\node[vertex] at (2.00000, -1.00000, 0.00000)     {};
\node[vertex] at (2.00000, 0.00000, -1.00000)     {};
\node[vertex] at (2.00000, 0.00000, 1.00000)     {};
\node[vertex] at (2.00000, 1.00000, 0.00000)     {};
\end{scope}

\begin{scope}[scale=0.9, xshift=7cm, x={(0.989872cm, -0.105320cm)},
	y={(0.095192cm, 0.989872cm)},
	z={(-0.105320cm, -0.095192cm)},
	scale=2.500000,
	back/.style={loosely dotted, thin},
	edge/.style={color=blue!95!black, thick},
	facetred/.style={fill=red!95!black,fill opacity=0.200000},
	facet/.style={fill=blue!95!black,fill opacity=0.200000},
	vertex/.style={inner sep=1pt,circle,draw=green!25!black,fill=green!75!black,thick,anchor=base}]

%
%
%% Coordinate of the vertices:
%%
\coordinate (-1.00000, 0.00000, 0.00000) at (-1.00000, 0.00000, 0.00000);
\coordinate (0.00000, -1.00000, 0.00000) at (0.00000, -1.00000, 0.00000);
\coordinate (0.00000, 0.00000, -1.00000) at (0.00000, 0.00000, -1.00000);
\coordinate (0.00000, 0.00000, 1.00000) at (0.00000, 0.00000, 1.00000);
\coordinate (0.00000, 1.00000, 0.00000) at (0.00000, 1.00000, 0.00000);
\coordinate (1.00000, 0.00000, 0.00000) at (1.00000, 0.00000, 0.00000);
%%
%%
%% Drawing edges in the back
%%
\draw[edge,back] (-1.00000, 0.00000, 0.00000) -- (0.00000, 0.00000, -1.00000);
\draw[edge,back] (0.00000, -1.00000, 0.00000) -- (0.00000, 0.00000, -1.00000);
\draw[edge,back] (0.00000, 0.00000, -1.00000) -- (0.00000, 1.00000, 0.00000);
\draw[edge,back] (0.00000, 0.00000, -1.00000) -- (1.00000, 0.00000, 0.00000);
%%
%%
%% Drawing vertices in the back
%%
\node[vertex] at (0.00000, 0.00000, -1.00000)     {};
%%
%%
%% Drawing the facets
%%
\fill[facet] (0.00000, 1.00000, 0.00000) -- (-1.00000, 0.00000, 0.00000) -- (0.00000, 0.00000, 1.00000) -- cycle {};
\fill[facet] (0.00000, 0.00000, 1.00000) -- (-1.00000, 0.00000, 0.00000) -- (0.00000, -1.00000, 0.00000) -- cycle {};
\fill[facet] (1.00000, 0.00000, 0.00000) -- (0.00000, 0.00000, 1.00000) -- (0.00000, 1.00000, 0.00000) -- cycle {};
\fill[facet] (1.00000, 0.00000, 0.00000) -- (0.00000, -1.00000, 0.00000) -- (0.00000, 0.00000, 1.00000) -- cycle {};
%%
%%
%% Drawing edges in the front
%%
\draw[edge] (-1.00000, 0.00000, 0.00000) -- (0.00000, -1.00000, 0.00000);
\draw[edge] (-1.00000, 0.00000, 0.00000) -- (0.00000, 0.00000, 1.00000);
\draw[edge] (-1.00000, 0.00000, 0.00000) -- (0.00000, 1.00000, 0.00000);
\draw[edge] (0.00000, -1.00000, 0.00000) -- (0.00000, 0.00000, 1.00000);
\draw[edge] (0.00000, -1.00000, 0.00000) -- (1.00000, 0.00000, 0.00000);
\draw[edge] (0.00000, 0.00000, 1.00000) -- (0.00000, 1.00000, 0.00000);
\draw[edge] (0.00000, 0.00000, 1.00000) -- (1.00000, 0.00000, 0.00000);
\draw[edge] (0.00000, 1.00000, 0.00000) -- (1.00000, 0.00000, 0.00000);
%%
%%
%% Drawing the vertices in the front
%%
\node[vertex] at (-1.00000, 0.00000, 0.00000)     {};
\node[vertex] at (0.00000, -1.00000, 0.00000)     {};
\node[vertex] at (0.00000, 0.00000, 1.00000)     {};
\node[vertex] at (0.00000, 1.00000, 0.00000)     {};
\node[vertex] at (1.00000, 0.00000, 0.00000)     {};
\end{scope}

\end{tikzpicture}
\end{figure}
The group $O_4(\mathbb{Z})$ acts on faces of it. In particular, its edges split into two orbits:
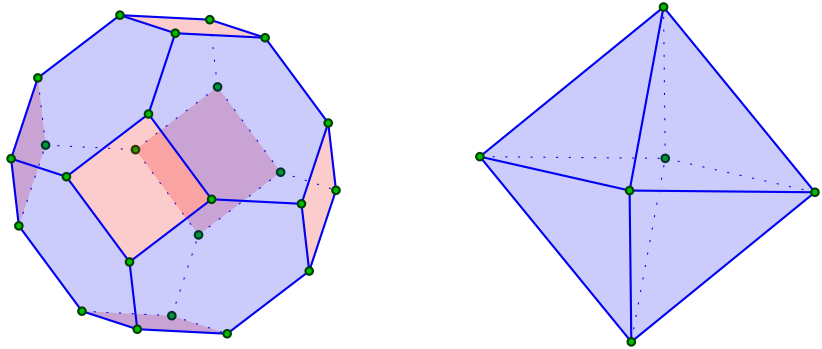
\begin{enumerate}
\item One orbit consists of edges between a square and an hexagon. There are $24$ of these, and we denote their common $\alpha$-value by $\alpha_1$.
\item The other orbit consists of edges between the hexagons. There are $12$ of these, and we denote their common $\alpha$-value by $\alpha_2$.
\end{enumerate}
Note that %there are no interior lattice points in any of the edges. So 
all edges have normalized volume $1$. By Equation \eqref{equ:coeff}, in addition to the knowledge that the linear coefficient of $i(\Pi_3,t)$ is $6$, we get an equation on $\alpha_1$ and $\alpha_2:$ 
$24\alpha_1 + 12\alpha_2 = 6.$

We need one more equation in order to find values of $\alpha_i$'s, 
%This is not enough to determine the values of $\alpha_i$'s, 
so we look at a deformation of $\Pi_3$. We can push away the square faces to get the polytope $\Perm(1,1,4,4)$, which after translating by $(1,1,1,1)$ and scaling by $3$ becomes the hypersimplex $\Delta_{2,4}$ or $\Perm(0,0,1,1)$. This is a regular octahedron shown on the right side of the figure above.
Notice that all the edges in this new polytope have the same normal cones as the edges in the second orbit above. Hence, they must have the same $\alpha$-value, which we have denoted $\alpha_2$. By Equation \eqref{equ:coeff}, in addition to the knowledge that the linear coefficient of $i(\Delta_{2,4},t)$ is $7/3$, we get our second equation 
$12\alpha_2 = 7/3.$ % \frac{7}{3}.$

Solving the two linear equations above, we obtain $\alpha_1=11/72$ and $\alpha_2 = 14/72.$
%\begin{align*}
%\alpha_1 &= 11/72,\\
%\alpha_2 &= 14/72.
%\end{align*} 
\end{ex}

Property (P5), in particular Lemma \ref{lem:symmetric}, does not hold for all the solutions for McMullen's formula. For example, the construction given by Pommersheim and Thomas in \cite{toddclass} depends on an ordering of a basis for the vector space, which means their construction is not symmetric about coordinates. 
This is the main reason why we work with the BV-$\alpha$-valuations for this paper.

\section{Generic permutohedron}\label{sec:generic} 

%The goal of this section is to prove Theorems \ref{thm:truelowdim} and \ref{thm:34coeff}.
Since the proof of Theorem \ref{thm:reduction} is completed in the last section, we will focus on the BV-$\alpha$-valuation arising from the regular permutohedron $\Pi_n,$ or any generic permutohedron. (See Remark \ref{rem:generic}.) 
We use the following setup. 
\begin{setup} \label{setup1}
  \begin{ilist}
\itm Let $\v = (v_1, v_2, \dots, v_{n+1})$ be a vector with strictly increasing entries, and consider the generic permutohedron 
%\[\Perm(\v) = \Perm(v_1, v_2, \dots, v_{n+1}).\]
$\Perm(\v) = \Perm(v_1, v_2, \dots, v_{n+1}).$
\itm Suppose $\alpha$ is a construction such that McMullen's formula \eqref{equ:exterior1} holds and it is symmetric about the coordinates (see Definition \ref{defn:symmetric}).
\end{ilist}
\end{setup}
It is clear that (i) covers all generic permutohedron, and the BV-$\alpha$-valuation is a special case of (ii). Under this setup, we will analyze Equation \eqref{equ:exteriorpermv} below for computing $\Lat(\Perm(\v))$ further, and derive a more combinatorial formula for computing $\alpha$-values arising from $\Perm(\v).$ 

Applying McMullen's formula to $P = \Perm(\v),$ we get
\begin{equation}\label{equ:exteriorpermv}
\Lat(\Perm(\v)) = \displaystyle \sum_{F: \textrm{ a face of $\Perm(\v)$}} \alpha(F, \Perm(\v))  \cdot  \nvol(F).
\end{equation}
Because of the symmetric properties of $\Perm(\v)$ and $\alpha,$ a lot of terms in the above summand coincide. Hence, it is natural to group them together as in Example \ref{ex:3d}. In order to do this, we need the following definition and proposition.
%Proposition \ref{prop:faceposet} below gives us a tool to achieve it. 
\begin{defn}
The symmetric group $\fS_{n+1}$ acts linearly on $\R^{n+1}$ by permuting the coordinates. Two subsets $A_1,A_2 \subset \R^{n+1}$ are said to be \emph{symmetric} if they lie in the same orbit, i.e. if there exist $\sigma\in \fS_{n+1}$ such that $\sigma(A_1)=A_2$. Since the action is orthogonal, two symmetric sets are congruent, in particular, they have the same volume (if measurable).
\end{defn}

%We begin reviewing some useful facts. 
The following results are well-known. 
%(See for example \cite[Chapter 7]{zie}.)
(See for example \cite[Propostion 2.2 of Chapter VI]{barvinokconvex}.)
\begin{prop}\label{prop:faceposet}
There is a one-to-one correspondence between ordered set partitions of $[n+1]$ and faces of $\Perm(\v)$ defined as follows:

For any ordered set partition ${\cP}= (P_1, P_2, \cdots, P_l)$ of $[n+1]$, the corresponding face is obtained by maximizing any linear functional given by a vector $\c \in \R^{n+1}$ with the property that 
\begin{alist}
\itm $c_i=c_j$ if $i$ and $j$ are both in $P_k$ for some $k$, and
  \itm $c_i<c_j$ if $i\in P_{k_1}$ and $j\in P_{k_2}$ with $k_1< k_2$.
\end{alist}
Let $m_i=|P_i|$. Then the corresponding face has dimension $n+1-l$ and it is congruent to
\[
\Perm(\v_{M_1})\times \Perm(\v_{M_2})\times \cdots \Perm(\v_{M_l}),
\]
where $\v_{M_i} = \left(v_j:\quad \displaystyle\sum_{k=1}^{i-1} m_k < j \leq  \sum_{k=1}^i m_k\right)$. In other words, $\v_{M_1}$ consists of the first $m_1$ entries of $\v=(v_1,\dots, v_{n+1}),$ $\v_{M_2}$ consists of the next $m_2$ entries, and so on.
%Abusing notation, we will freely identify the ordered partition and the face.

We call the ordered tuple $\m:=(m_1,m_2,\cdots,m_l)$ \emph{the composition of $\cP$}. 
\end{prop}
\begin{ex}\label{ex:compositions} 
  Let $n=5$ and consider the ordered set partition $\cP = \left(\left\{1,4,6\right\}, \left\{2,5\right\}, \left\{3\right\}\right)$. Then the composition of $\cP$ is $(3,2,1).$ 
  
  The face of $\Perm(\v)$ corresponding to $\cP$ is the face which optimizes any linear functional $\c=(c_1,c_2,c_3,c_4,c_5,c_6)$ with $c_1=c_4=c_6<c_2=c_5<c_3$. 
  In order to figure out this corresponding face, we look for vertices of $\Perm(\v)$ optimizing such a functional $\c.$ One sees that $v_1$,  $v_2$ and $v_3$ should be in positions 1, 4 and 6 of these vertices, $v_4$ and $v_5$ in positions 2 and 5, and $v_6$ in position 3. Therefore, the desired vertices are  \[ \left\{ \left(v_{\mu(1)}, v_{\tau(1)}, v_6, v_{\mu(2)}, v_{\tau(2)}, v_{\mu(3)}\right) \ : \ \mu \in \fS_{\{1,2,3\}}, \tau \in \fS_{\{4,5\}} \right\}.
  \]  
 % For example, $(v_1,v_5,v_6,v_3,v_4,v_2)$, $(v_1,v_4,v_6,v_3,v_5,v_2)$, and $(v_3,v_5,v_6,v_2,v_4,v_1)$ are three of the vertices that are in this corresponding face.
  Hence, we conclude that the face that is corresponding to the ordered set partition $\cP= \left(\left\{1,4,6\right\}, \left\{2,5\right\}, \left\{3\right\}\right)$ is congruent to 
%\[
$\Perm(v_1,v_2,v_3)\times\Perm(v_4,v_5)\times \Perm(v_6).$
%\]
\end{ex}

By Proposition \ref{prop:faceposet}, two faces of $\Perm(\v)$ are in the same orbit, i.e. they are symmetric, if and only if their corresponding ordered set partitions have the same composition. Therefore, the orbits of the $\fS_{n+1}$-action on the faces of $\Perm(\v)$ are indexed by compositions $\m$ of $n+1$. 
We denote the orbit corresponding to the composition $\m$ by $\O_n(\m).$

Furthermore, under Setup \ref{setup1}, the $\alpha$-construction is symmetric about the coordinates. Hence, for any fixed $\m,$ the value $\alpha(F, \Perm(\v))$ is a constant on $\O_n(\m),$ and thus we can define $\alpha_n(\m)$ to be this constant.

Finally, a canonical representative of $\O_n(\m)$ is chosen as below.
\begin{defn}
Let $\m = (m_1, m_2, \dots, m_l)$ be a composition of $n+1$. Define an ordered set partition ${\cP}(\m)=\big({\cP}(\m)_i\big)$ where 
\[
{\cP}(\m)_i=\left[\displaystyle\sum_{k=1}^{i-1} m_k +1, \sum_{k=1}^i m_k\right].
\]
In other words, the first subset $\cP(\m)_1$ consists of the first $m_1$ positive integers, the second subset $\cP(\m)_2$ consists of the next $m_2$ positive integers, and so on. 

Then we define $F_\m$ to be the face corresponding to the ordered set partition ${\cP}(\m)$ under the bijection given in Proposition \ref{prop:faceposet}.
\end{defn}
\begin{ex}\label{ex:compositions2}
Let $n=5$ and $\m=(3,2,1)$. Then ${\cP}(\m)= \left(\left\{1,2,3\right\}, \left\{4,5\right\}, \left\{6\right\}\right),$ and 
\begin{equation}\label{eqn:exFm}
  F_\m = \conv\left\{ \left(v_{\mu(1)}, v_{\mu(2)}, v_{\mu(3)}, v_{\tau(1)}, v_{\tau(2)}, v_6\right) \ : \ \mu \in \fS_{\{1,2,3\}}, \tau \in \fS_{\{4,5\}} \right\}. \end{equation}
\end{ex}

%\begin{defn}\label{def:compositionset}
%Given composition $m=(m_1,m_2,\cdots, m_k)$ of $n+1$, define the ordered partition 
%\end{defn}
%

Applying the above discussions to \eqref{equ:exteriorpermv}, we get
\begin{equation}\label{eqn:compositions}
\Lat(\Perm(\v))=\displaystyle \sum_{\textrm{$\m:$ composition of $n+1$}} |\O_n(\m)| \cdot \alpha_n(\m) \cdot \nvol(F_\m).
\end{equation}
%where $O_n(\m)$ is the size of the orbit, i.e. the number of ordered set partitions of $[n+1]$ with composition $\m$.
Note that one of the terms in the above formula can be explicitly described: For a fixed $\m=(m_1,\cdots, m_l)$, the number of faces in $\O_n(\m)$ is equal to the number of ordered set partitions whose compositions are $\m.$ Thus,
\begin{equation}\label{eqn:orbits}
|\O_n(\m)| = \binom{n+1}{m_1,m_2,\cdots, m_l}.
\end{equation}

%Next, we investigate properties of the face $F_\m$. 
It is easy to see that $F_\m$ is always adjacent to the vertex $\v = (v_1, \dots, v_n, v_{n+1}).$ In fact, we show below that every face adjacent to $\v$ arises as $F_\m$ for a unique $\m.$ Note that the vertex cone of $\Perm(\v)$ at $\v$ is spanned by the following $n$ vectors:
\[ \e_1 - \e_2, \e_2 - \e_3, \dots, \e_n - \e_{n+1}.\]
Hence, subsets of these $n$ vectors are in one-to-one correspondence to faces of $\Perm(\v)$ that are adjacent to $\v.$ 
%This motivates the following definition.

%\begin{defn}\label{defn:comp2subset}
 % Let $\m$ be a composition of $n+1,$ and $F_\m$ the face of $\Perm(\v)$ that is associated to $\m.$ Define $S = \cS(\m)$ to be the subset of $[n]$ such that   \[ \textrm{affine span of }F_\m = \v + \textrm{span}\{ \e_i - \e_{i+1} \ : \ i \in S\}.\]
%\end{defn}
\begin{lem}\label{lem:bij} 
  Let $\m$ be a composition of $n+1,$ and $F_\m$ the face of $\Perm(\v)$ that is associated to $\m.$ Define $S = \cS(\m)$ to be the subset of $[n]$ such that   \[ \aff(F_\m) = \v + \textrm{span}\{ \e_i - \e_{i+1} \ : \ i \in S\}.\]
  
  The map $\cS$ is a bijection from compositions of $n+1$ to subsets of $[n].$ Hence, $\{ F_\m \ : \ \m \text{ is a composition of $n+1$}\}$ consists of all faces of $\Perm(\v)$ that are adjacent to $\v.$
 \end{lem}

\begin{ex}
  Let $n=5$ and $\m=(3,2,1)$. Then $F_\m$ is given by \eqref{eqn:exFm}. One checks that 
\[ \aff(F_\m) = \v + \textrm{span}\{ \e_1 - \e_2, \e_2-\e_3, \e_4 - \e_5\}.\]
Hence, $S = \cS(\m) = \{1, 2, 4\}.$
\end{ex}

%Note that both compositions of $n+1$ and subsets of $[n]$ are counted by the same number, $2^n.$ So it is natural to ask whether the map $\m \mapsto \cS(\m)$ is a bijection between these two families of objects. Indeed, we have the following lemma.

\begin{proof}
We can define an inverse to $\cS$ in the following way: Suppose $S$ is a subset of $[n].$ Let $T = [n+1] \setminus S.$ Suppose $T = \{t_1 < t_2 < \cdots < t_l\}.$ Then one verifies that 
\[ S \mapsto \m := (t_1, t_2-t_1, \dots, t_l - t_{l-1})\]
is an inverse to $\cS,$ completing the proof. 
\end{proof}

%The above lemma tells us that the term $F_\m$ appearing in the summand of Formula \eqref{eqn:compositions} can also be understood as faces of $\Perm(\v)$ that are adjacent to the vertex $\v.$ Furthermore, 
By the above lemma, we may abuse notation and use subsets of $[n]$ to index Formula \eqref{eqn:compositions}:
\begin{equation}\label{eqn:subsets}
\Lat(\Perm(\v)) = \displaystyle \sum_{S\subseteq[n]} |\O_n(S)| \cdot \alpha_n(S) \cdot \nvol(F_S).
\end{equation}
Note that $\dim(F_S)=|S|$.
%In fact, for arguments we will carry out in both Section \ref{sec:evidence} and Section \ref{sec:unique}, it is more convenient to use subsets of $[n]$ as our indexing. 

\subsection*{Partial results on conjectures}
%We will explore consequences of Formula \eqref{eqn:subsets} further in the next section, and will devote the rest of the section to proof of Theorems \ref{thm:truelowdim} and \ref{thm:34coeff}. As we've discussed at the beginning of this section, we only need to verify BV-$\alpha$-positivity for the regular permutohedron $\Pi_n.$ Therefore, we now assume $\v = (1, 2, \dots, n, n+1)$ and $\alpha$ is the BV-$\alpha$-valuation, and determine whether $\alpha_n(S)$ is positive for the situations we consider.
We will explore consequences of Formula \eqref{eqn:subsets} further in Section \ref{sec:unique}, and will devote the rest of this section to proving Theorem \ref{thm:partial}, providing partial results on our main conjectures.
%only focus on computing $\alpha_n(S)$ for some special cases in this section. 
Clearly, in order to prove Theorem \ref{thm:partial}, we just need to verify the following two statements respectively, assuming $\alpha$ is the BV-$\alpha$-valuation:
\begin{align}
  \alpha_n(S) > 0, \quad &\forall S \subseteq [n], \forall n \le 6, \label{equ:smalldim} \\
  \alpha_n(S) >0, \quad &\forall S \subset [n], |S| = n-2, n-3. \label{equ:34coeff}
\end{align}
Applying Formula \eqref{equ:defnalpha} to our situation, we get
\begin{equation}\label{equ:alphaS}
  \alpha_n(S) = \Psi_{\R^{n+1}/\lin(F_S)}([(\ncone(F_S, \Pi_n))^\circ]). %,
\end{equation}

Recall we discussed briefly how to compute Berline-Vergne's $\Psi_W$ in \S \ref{subsec:excmp}. In particular, we discussed that $\Psi_W([C])$ is relatively easier to compute if $C$ is a unimodular cone with respect to the lattice $\Lambda(W).$
The following result on $\ncone(F_S, \Pi_n)^\circ$ shows that the computation of $\alpha_n(S)$ fits into this situation.

\begin{lem}\label{lem:desncone} Let $S \subseteq [n]$ and $V = \R^{n+1}.$ Then
	\[ \ncone_{V}(F_S, \Pi_n) =\{ \c =(c_1, \dots, c_{n+1}) \in V^* \ | \ c_i = c_{i+1} \ \forall i \in S \text{ and } c_i \le c_{i+1} \ \forall i \not\in S\}. \]
	Furthermore, if $[n] \setminus S =\{ i_1 < i_2 < \cdots < i_\ell\}$ and for each $1 \le j \le \ell,$ we define a vector 
	\begin{equation}
		R_{i_j} = \left( \underbrace{0, \cdots, 0}_{i_{j-1}}, \underbrace{\frac{1}{i_j - i_{j-1}}, \cdots, \frac{1}{i_j-i_{j-1}}}_{i_j-i_{j-1}}, \underbrace{\frac{-1}{i_{j+1}-i_{j}}, \cdots, \frac{-1}{i_{j+1}-i_j}}_{i_{j+1}-i_j}, \underbrace{0, \dots, 0}_{n+1-i_{j+1}} \right),
		\label{equ:defnRs}
	\end{equation}
	where by convention $i_0=0$ and $i_{\ell+1}=n+1,$
	then $\ncone(F_S, \Pi_n)^\circ$ is spanned by $R_{i_j}$. Moreover, $\{ R_{i_j}\}$ can be extended to a basis of the lattice $\Lambda(\R^{n+1}/\lin(F_S))$.
\end{lem}

\begin{proof}
	The formula for $\ncone_V(F_S, \Pi_n)$ follows from Proposition \ref{prop:faceposet} and Lemma \ref{lem:bij}. The second conclusion then follows from a direct calculation. 

	By Lemma \ref{lem:bij}, $\lin(F_S) = {\rm span}(\e_i - \e_{i+1} \ : \ i \in S).$ Note that $\{ \e_1 \} \cup \{ \e_i - \e_{i+1} \ : \ i \in [n]\}$ is a basis for $\Z^{n+1}.$ Hence, the orthogonal projections of $\{ \e_1 \} \cup \{ \e_{i_j} - \e_{i_j+1} \ : \ 1 \le j \le \ell \}$ onto $\lin(F_S)^\perp$ is a basis for $\Lambda(\R^{n+1}/\lin(F_S))$. One checks for each $j,$ the orthogonal projection of $\e_{i_j} - \e_{i_j+1}$ is $R_{i_j}.$ Thus, the last assertion follows.
\end{proof}
 
Applying Lemmas \ref{lem:2dim} and \ref{lem:3dim} to the above lemma with $\ell = 2$ and $3$, we obtain precise formulas for $\alpha_n(S)$ when $|S|=n-2$ or $n-3.$
\begin{cor}\label{cor:alphaS34}
 Suppose $S \subseteq [n]$.
 \begin{ilist}
   \itm If $[n] \setminus S =\{i,j\}$ with $i < j,$ then
 % \begin{equation}\label{equ:alphaS3}
$ \ds \alpha_n(S) = \frac{1}{4} - \frac{1}{12}\left(\frac{i}{j} + \frac{n+1-j}{n+1-i}\right).$
%\end{equation}
\itm If $[n] \setminus S =\{i,j,k\}$ with $i < j<k,$ then
%\begin{equation}\label{equ:alphaS4}
$\ds \alpha_n(S) = \frac{1}{8} - \frac{1}{24}\left(\frac{i}{j} +1+ \frac{n+1-k}{n+1-j}\right).$
%\end{equation}
 
 \end{ilist}

\end{cor}

\begin{proof}[Proof of Theorem \ref{thm:partial}]
  We are able to verify that \eqref{equ:smalldim} is true by directly applying Berline-Vergne's construction $\Psi_W$ and using \eqref{equ:alphaS} and Lemma \ref{lem:desncone}. (We omit all the data to save space. For interested readers, please see examples in \cite[Section 5.1]{casliu-BValpha-v1}.) 
So Part (1) of the theorem follows.

Next, it is easy to check that $\alpha_n(S)$ are positive in both formulas provided in Corollary \ref{cor:alphaS34}. Hence, \eqref{equ:34coeff} and Part (2) of the theorem follow.
\end{proof}

  \begin{rem}
    By Remark \ref{rem:upton}, if $\alpha_6(S)$ is positive for all $S \subseteq [6],$ then we immediately have $\alpha_n(S) > 0$ for all $n < 6$ and any $S \subseteq [n].$ Hence, the proof of Theorem \ref{thm:partial}/(1) can be reduced to proving $\alpha_6(S) > 0$ only. 

    Similarly, 
    % Similar as to the proof of Theorem \ref{thm:partial}/(1), 
    it is not necessary to show $\alpha_n(S)$ for both $|S|=n-2$ and $|S|=n-3$ to complete a proof of Theorem \ref{thm:partial}/(2). In fact, it follows from Remark \ref{rem:upton} that if $\alpha_n(S) >0$ for all $n$ and $S$ of size $n-3,$ then $\alpha_n(S)>0$ for all $n$ and $S$ of size $n-2.$ %We again include the cases of $|S|=n-2$ to provide more data of $\alpha_n(S).$
  \end{rem}

\section{Uniqueness}\label{sec:unique}
In this section, we take a different point of view and investigate the uniqueness of the $\Psi_W$/$\alpha$ constructions for McMullen's formula. 
%
%Here we prove that, contrary to the origins of the Exterior Formula, for the faces of generalized permutohedra the $\alpha$ values are uniquely determined once we required them to be invariant under symmetric action.
%
We will apply the mixed valuation theories introduced in \S \ref{subsec:mixed} to Minkowski sums of \emph{hypersimplices}.
\begin{defn}
  The \emph{hypersimplex} $\Delta_{k,n+1}$ is defined as %$\Perm(\v)$ where 
  \[ \Delta_{k,n+1} = \Perm ( \underbrace{0,\cdots,0}_{n+1-k}, \underbrace{1,\cdots,1}_{k} ).\]
\end{defn}
The main goal of this part is to prove Theorem \ref{thm:uniqueness0}, which will be rephrased below as Theorem \ref{thm:uniqueness} stating that the $\alpha$-values of faces of $\Perm(\v)$ are uniquely determined as a scalar of mixed valuation of hypersimplices if we require $\alpha$ and $\v$ to be given under Setup \ref{setup1}. Furthermore, as a consequence of this result, we give an equivalent statement of Conjecture \ref{conj:alphas} in Corollary \ref{cor:equimixed}.

As in Setup \ref{setup1}, we consider the generalized permutohedron $\Perm(\v) = \Perm(v_1,v_2,\cdots, v_{n},v_{n+1})$ with $v_1<v_2<\cdots<v_{n}<v_{n+1}$.
We have the following expression for $\Perm(\v)$ as Minkowski sum \cite[Section 16]{bible}. 
\[
\Perm(\v) = w_1\Delta_{1,n+1} + w_2\Delta_{2,n+1} + \cdots + w_{n}\Delta_{n,n+1},
\]
where 
\begin{equation}\label{equ:defnwi}
w_i := v_{i+1}-v_i \text{ for } i=1,2,\dots,n.
\end{equation}
(The $w_i$'s are actually lengths of edges of $\Perm(\v).$ But this is not relevant to our discussion.)
%The information about the normalized volumes of all faces is given by the lengths of the edges, which are:
%\begin{eqnarray*}
%w_1&:=&v_2-v_1,\\
%w_2&:=&v_3-v_2,\\
%&\vdots & \\
%w_{n}&:=& v_{n+1}-v_{n}.
%\end{eqnarray*}

Using the results on mixed volumes -- Theorems \ref{thm:mixedvol} and \ref{thm:mixvolprop} -- we have the following: 
\begin{lem}\label{lem:voledges}
  The normalized volume of $\Perm(\v)$ is a homogeneous polynomial in $w_i$'s with strictly positive coefficients. 
\end{lem}

In \cite{bible}, the coefficients of the above homogeneous polynomial are called \emph{mixed Eulerian numbers}, and some basic properties are established. One of the properties is the following: 
\begin{lem}\label{lem:squarefree}
  The coefficient of $w_1w_2\cdots w_n$, the unique squarefree monomial, in the homogeneous polynomial assumed in Lemma \ref{lem:voledges} is $n!$
\end{lem}
 Note that Lemma \ref{lem:squarefree} says that 
 $\sum_{\sigma \in \fS_n} \ml^n(\Delta_{\sigma(1), n+1}, \Delta_{\sigma(2),n+1}, \dots, \Delta_{\sigma(n),n+1}) = n!,$ which by Lemma \ref{lem:mixval}/(i) is equivalent to
	 \begin{equation} \ml^n(\Delta_{1,n+1},\Delta_{2,n+1},\dots, \Delta_{n,n+1}) = 1.
\label{equ:mixedhypersimplices}
\end{equation}

Recall in Section \ref{sec:generic}, we associate a face $F_\m$ of $\Perm(\v)$ to any composition $\m$ of $n+1,$ establish a bijection $\cS$ from $\m$ to subsets $S$ of $[n]$, and rewrite $F_\m$ as $F_S.$ We have the following result on the normalized volume of $F_S.$ 

\begin{prop}\label{prop:nvolumes}
  Suppose $P=\Perm(\v)$ and $S\subseteq [n]$. Let $F_S$ be the corresponding face of $P$ as defined in Section \ref{sec:generic}, and $\m=(m_1,\cdots, m_{l}) := \cS^{-1}(S)$ is the composition in bijection to $S$. Then $\nvol(F_S)$ is a homogeneous polynomial in $\{w_i \ : \ i \in S\},$ whose coefficient of $\prod_{i \in S} w_i$-- the unique squarefree monomial -- is 
  \begin{equation}\label{equ:CnS}
C_n(S) :=(m_1-1)!(m_2-1)!\cdots (m_{l}-1)!,
\end{equation}
\end{prop}

\begin{proof} Suppose $[n+1]\setminus S =\{t_1 < t_2 < \dots < t_{l-1} < t_l =n+1\},$ and by convention let $t_0=0.$ Note that by the proof of Lemma \ref{lem:bij}, we have $\sum_{k=1}^i m_k = t_i$ for each $i.$ 

  By Proposition \ref{prop:faceposet}, the face $F_S$ is congruent to
 $ 
\Perm(\v_{M_1})\times \Perm(\v_{M_2})\times \cdots \Perm(\v_{M_l}),
$
%where $\v_{M_i} = \left(v_j:\quad \displaystyle\sum_{k=1}^{i-1} m_k < j \leq  \sum_{k=1}^i m_k\right)$.
where $\v_{M_i} = \left(v_j \ : \ \displaystyle t_{i-1} < j \leq  t_i \right)$.
Hence, $\nvol(F_S) = \prod_{i=1}^l \nvol(\Perm(\v_{M_i})).$
Let 
%\[ T_i := \left\{ j \ : \ \sum_{k=1}^{i-1} m_k < j < \sum_{k=1}^i m_k \right\}.\]
\[ T_i := \left\{ j \ : \ t_{i-1} < j < t_i \right\}.\]
Then by Lemmas \ref{lem:voledges} and \ref{lem:squarefree}, the normalized volume of $\Perm(\v_{M_i})$ is a homogeneous polynomial in $\{ w_j \ : \ j \in T_i\},$ and the coefficient of $\prod_{j \in T_i} w_j$ -- the unique squarefree monomial -- in this homogeneous polynomial is $(m_i-1)!.$ Therefore, the conclusion follows from the observation that $S = \cup_{i=1}^l T_i.$ 
\end{proof}

The following is the main result of this section, which is an expanded version of Theorem \ref{thm:uniqueness0}.
\begin{thm}\label{thm:uniqueness}
%Suppose $\alpha$ is a construction such that the exterior formula \eqref{equ:exterior1} holds and it is symmetric about the coordinates (see Definition \ref{defn:symmetric}). 
  Suppose $\alpha$ and $\v$ are as in Setup \ref{setup1} and $C_n(S)$ is defined as in \eqref{equ:CnS}.  Then the $\alpha$ values of faces of $\Perm(\v)$ are uniquely determined. More specifically, if $S = \{s_1, s_2, \dots, s_k\},$ we have
  \begin{align}
\alpha_n(S) =& \ \frac{1}{C_n(S)|\O_n(S)|}k!\ml^{k}(\Delta_{s_1,n+1},\Delta_{s_2,n+1},\cdots,\Delta_{s_k,n+1}) \label{equ:alpha0}  \\
=& \ \frac{m_1\cdot m_2 \cdots  m_{l}}{(n+1)!} k!\ml^{k}(\Delta_{s_1,n+1},\Delta_{s_2,n+1},\cdots,\Delta_{s_k,n+1}). \label{equ:alpha}
\end{align}

In particular the above formula applies to the BV-$\alpha$-valuation.
\end{thm}

One sees that the above theorem gives a connection between the $\alpha$ arising from the regular permutohedron and the mixed lattice point valuation $\ml^k$ on hypersimplices. Therefore, we have the following:
\begin{cor}\label{cor:equimixed}
  The following two statements are equivalent:
\begin{enumerate}
  \item For any $S =\{s_1, \dots, s_k\} \subseteq [n],$ we have
  $\ds \ml^{k}(\Delta_{s_1,n+1},\Delta_{s_2,n+1},\cdots,\Delta_{s_k,n+1}) > 0.$
\item The regular permutohedron $\Pi_n$ is BV-$\alpha$-positive.
\end{enumerate} 
  \end{cor}

\begin{proof}[Proof of Theorem \ref{thm:uniqueness}]
  Let $w_i$ be defined as in \eqref{equ:defnwi}.
Theorem \ref{thm:mixedLat} or Theorems \ref{thm:mixed} and \ref{thm:decomposeLat} tell us that the number of lattice points in
\[\Perm(\v) =  w_1\Delta_{1,n+1} + w_2\Delta_{2,n+1} + \cdots + w_{n}\Delta_{n,n+1}\] 
is a polynomial in the $w_i$ variables.
We denote this polynomial by $E = E(w_1, w_2,\dots, w_n).$ We focus on the coefficient of squarefree monomials $w_S :=\prod_{i\in S}w_i$ in $E.$ On the one hand, by \eqref{equ:mixed} and Lemma \ref{lem:mixval}/(i), this coefficient is equal to
\begin{equation}\label{eqn:coeff1}
  k!\ml^{k}(\Delta_{s_1,n+1},\Delta_{s_2,n+1},\cdots, \Delta_{s_k,n+1}).
\end{equation}
Next by Equation \eqref{eqn:subsets}, we have
\[
E(w_1, \dots, w_n) = \displaystyle \sum_{S\subset[n]} |\O_n(S)| \cdot \alpha_n(S) \cdot \nvol(F_S).
\]
(Proposition \ref{prop:nvolumes} guarantees that the right hand side of the above expression is indeed polynomial on the $w_i$ variables.)
Note that according to Proposition \ref{prop:nvolumes}, the only contribution to the monomial $w_S=\prod_{i\in S}w_i$ in the summand above is the term corresponding to $S$, and it is given by $C_n(S)$. Using these, we conclude that the coefficient of $w_S$ in $E(w_1,\dots, w_n)$ is
\begin{equation}\label{eqn:coeff2}
\alpha_n(S) \cdot |\O_n(S)| \cdot C_n(S)
\end{equation}
Our two expressions, \eqref{eqn:coeff1} and \eqref{eqn:coeff2}, for the coefficient of $w_S$ in $E$ have to agree. Hence, \eqref{equ:alpha0} follows. 

Finally, \eqref{equ:alpha} follows from \eqref{equ:CnS}, \eqref{eqn:orbits}, and the bijection defined in Lemma \ref{lem:bij}.
\end{proof}

%Proposition \ref{prop:nvolumes} and Equation \eqref{eqn:orbits} allow us to give an explicit formula for $|\O_n(S)|C_n(S)$ using the bijection defined in Lemma \ref{lem:bij}.
%\begin{equation}\label{equ:OC}
%\frac{1}{|\O_n(S)|C_n(S)} = \frac{m_1\cdot m_2 \cdots  m_{l}}{(n+1)!}.
%\end{equation}
%Hence, we rewrite Equation \eqref{equ:alpha0} as
%  \begin{equation}\label{equ:alpha}
%\alpha_n(S) = \frac{m_1\cdot m_2 \cdots  m_{l}}{(n+1)!} k!\ml^{k}(\Delta_{s_1,n+1},\Delta_{s_2,n+1},\cdots,\Delta_{s_k,n+1}).
%\end{equation}

Formula \eqref{equ:alpha} allows us to compute some examples of $\alpha_n(S)$.

%\begin{ex}
%	Consider $S=[n]$ as a subset of $[n+1]$. Its corresponding composition is $\m = \cS^{-1}(S) = (n+1)$, and $F_S = \Perm(\v).$ By Equation \eqref{equ:leading}, we have $\alpha_n(S) = 1.$
%	Therefore, Formula \eqref{equ:alpha} gives: 
%	\[ \frac{n+1}{(n+1)!}n!\ml^n(\Delta_{1,n+1},\Delta_{2,n+1},\dots, \Delta_{n,n+1}) = 1,\]
%	agreeing with \eqref{equ:mixedhypersimplices}.
%\end{ex}

\begin{ex}\label{ex:compLat^2}
	Consider $n=3$ and $S =\{1,3\} \subseteq [3].$ The corresponding composition is $\m = (2,2)$. Applying \eqref{equ:alpha}, we get 
	\[ \alpha_3(\{1,3\}) = \frac{2\cdot 2}{24} \ 2! \ml^2(\Delta_{1,4}, \Delta_{3,4}).\] %= \frac{1}{3} \ml^2(\Delta_{1,4}, \Delta{3,4}).\]
	Furthermore, Theorem \ref{thm:mobius} provides a way to compute mixed valuations:
	\[ 2!\ml^2(\Delta_{1,4},\Delta_{3,4}) = \Lat^2(\Delta_{1,4}+\Delta_{3,4})-\Lat^2(\Delta_{1,4}) - \Lat^2(\Delta_{1,4}).\]
	By the comment after Theorem \ref{thm:decomposeLat}, for any polytope $\Lat^r(P)$ is the coefficient of $t^r$ in the Ehrhart polynomial $i(P,t).$ Hence, we can figure out the terms in the right hand side of the above equality by computing corresponding Ehrhart polynomials:
\begin{eqnarray*}
i\left(\Delta_{14}+\Delta_{34},t\right) &=& \frac{10}{3}t^3 + 5t^2 + \frac{11}{3}t + 1,   \\ 
i\left(\Delta_{14},t\right) &=& \frac{1}{6}t^3 + t^2 + \frac{11}{6}t +1, \\
i\left(\Delta_{34},t\right) &=& \frac{1}{6}t^3 + t^2 + \frac{11}{6}t +1.
\end{eqnarray*} 
Therefore,
%\begin{align*}
$2!\ml^2(\Delta_{1,4},\Delta_{3,4}) =  5 - 1  - 1 =3,$ 
%\intertext{and} 
and
$\alpha_3(\{1,3\}) = \ds \frac{2 \cdot 2}{24} \cdot 3 = \frac{1}{2},$
%\end{align*}
which agrees with Lemma \ref{lem:01dim} since $F_{\{1,3\}}$ is a facet. 
\end{ex}

%\begin{ex}
%	Consider $S =\{i\}$. Then $F_S$ is an edge, and the corresponding composition $\m$ consists of one copy of $2$ and $n-1$ copies of $1$. Hence,
%\[
%\alpha_n(\left\{i\right\}) = \frac{2}{(n+1)!}\Lat^1(\Delta_{i,n+1}).
%\]
%We now use the above formula to compute $\alpha_5(\{i\})$ which are $\alpha(E, \Pi_5)$ for edges $E$ of $\Pi_5.$
%Again, $\Lat^1(\Delta_{i,6}$ is the linear coefficient of the Ehrhart polynomial of $\Delta_{i,6}.$ So we compute these polynomials:
%\begin{eqnarray*}
%i(\Delta_{1,6},t) &=& \frac{1}{120}t^5 + \frac{1}{8}t^4 + \frac{17}{24}t^3 + \frac{15}{8}t^2 + \frac{137}{60}t + 1,\\ 
%i(\Delta_{2,6},t) &=& \frac{13}{60}t^5 + \frac{3}{2}t^4 + \frac{47}{12}t^3 + 5t^2 + \frac{101}{30}t + 1,\\
%i(\Delta_{3,6},t) &=& \frac{11}{20}t^5 + \frac{11}{4}t^4 + \frac{23}{4}t^3 + \frac{25}{4}t^2 + \frac{37}{10}t + 1.
%\end{eqnarray*}
%Therefore,
%\begin{eqnarray*}
%\alpha_5(\left\{1\right\}) = \frac{2}{6!}\Lat^1(\Delta_{1,6}) &=& \frac{2}{720}\frac{137}{60} = \frac{137}{21600}, \\
%\alpha_5(\left\{2\right\}) = \frac{2}{6!}\Lat^1(\Delta_{2,6}) &=& \frac{2}{720}\frac{101}{30} = \frac{101}{10800},\\
%\alpha_5(\left\{3\right\}) = \frac{2}{6!}\Lat^1(\Delta_{3,6}) &=& \frac{2}{720}\frac{37}{10} = \frac{37}{3600},
%\end{eqnarray*}
%agreeing with Example \ref{ex:pi5}. 
%\end{ex}

\section{Further Questions and Remarks}\label{sec:further}

We finish the article with a brief description of other progress we've made on proving our conjectures and a discussion on questions/problems arising from this paper.

\subsection*{Other results} In addition to the results presented in this paper, we have two other related results, which are omitted because they are less important than those appeared in the paper, and we want to keep the paper within a reasonable length.

The first one was another partial result on our strong conjecture (Conjecture \ref{conj:alphas}).
%By Theorem \ref{thm:uniqueness0} the BV-$\alpha$ values are uniquely determined for permutohedra. 
Recall in Example \ref{ex:3d} we found the $\alpha$-values of two kinds of edges of $\Pi_3$ by setting up a triangular linear system. 
Using similar strategy, we can set up an explicit linear system for $\alpha$-values of edges of $\Pi_n$ for any $n.$ Since solving linear systems is very fast, we can find $\alpha$-values of edges of $\Pi_n$ quickly for $n$ that is not too large. For example, we computed $\alpha$-values of edges of $\Pi_{500}$ and verified that they were all positive. By Remark \ref{rem:upton}, this implies that our strong conjecture is true for edges of generalized permutohedra of dimension up to $500$. (See \cite{thesis} for details.) 
Even though it will be easy for us to push the number $500$ to a much larger number by solving linear systems explicitly, a systematic way to show all solutions are positive for all $n$ will be more desirable.
%The problem with this method in general lies in the difficulty of having a family of polytopes for which the Ehrhart coefficients can be computed in some effective way.

The second result is another equivalent statement to Conjecture \ref{conj:alphas} in addition to the equivalent statement in terms of mixed lattice point valuations provided in Corollary \ref{cor:equimixed}. 
The Berline-Vergne's construction gives one way to write the Todd class of the permutohedral variety in terms of the toric invariant cycles. We can show that if there is \textbf{any} way of writing such class as a positive combination of such cycles, then the BV-$\alpha$-valuation is one of them. (See \cite[Proposition 7.2]{casliu-BValpha-v1} or \cite{thesis}.)
This is important since there are other constructions that may work. For instance, if there is an appropriate choice of flags in Pommersheim-Thomas method that yields positive values, then it will prove Conjecture \ref{conj:alphas}. %This can be found on an earlier version of this paper \cite{casliu-BValpha-v1}.\\

\subsection*{Questions}
Naturally, the main question left is still whether Conjecture \ref{conj:alphas}, or any of its equivalences, is true. Other than that, the following questions may be of further interest.
\begin{enumerate}
\item \textbf{Uniqueness of the BV-$\alpha$-valuation.}
\begin{itemize}
  \item Can we generalize Theorem \ref{thm:uniqueness} to other families of polytopes that come from a certain normal fan?
    In other words, is there any other normal fan $\Sigma,$ such that the BV-$\alpha$-values arising from polytopes whose normal fan is $\Sigma$ are uniquely determined?
    %Is there any other fan for which the BV $\alpha$'s are uniquely determined?
\item More importantly, is the BV-$\alpha$-valuation the unique solution to McMullen's formula that is a valuation and is invariant under permutations of coordinates?
\end{itemize}
\item \textbf{Compute the BV-$\alpha$-valuation for some specific polytopes.} For the case of generalized permutohedra, we believe that knowledges of $\alpha$-values on hypersimplices would be a very useful step in understanding the general case.
We note that there seems to be very few examples explicitly computed in the literature.
\end{enumerate}
It is worth mentioning that in a recent paper \cite{gaku}, the author describes a recursive way to compute mixed Eulerian numbers, which are mixed volumes. The method extends to a recursive way of computing the mixed lattice point valuations of hypersimplices, but so far we cannot prove positivity that way. 

Lastly, we would like to mention a related idea. 
%We already discussed that there are other solutions to McMullen's formula. 
The \emph{exterior angle} $\gamma(F,P)$ of $P$ at $F$ is the ``intrinsic measure'' of $\ncone_V(F,P)$. 
It is clear from the definition of exterior angles, $\gamma$ is symmetric about the coordinates. 
Moreover, the following result on $\gamma(F,P)$ indicates that $\gamma$ is a partial solution to McMullen's formula.
\begin{thm}[Corollary 7.8 of \cite{BarviPom}]
Let $P\subset \R^D$ be an integral zonotope, i.e. a Minkowski sum of elements in $\Z^D$, then we have
\[
\Lat(P) = \displaystyle \sum_{F: \textrm{ a face of $P$}} \gamma(F,P)
\ \nvol(F)
\]
\end{thm} 
The above theorem applies to the regular permutohedron; however, unfortunately it does \textbf{not} apply to its deformations, not even for generic permutohedra. So this construction $\gamma$ doesn't satisfy Setup \ref{setup1}. 
Our arguments in Section \ref{sec:unique} depend on the fact that McMullen's formula holds for all generic permutohedra, and thus won't hold for the exterior angle construction. Indeed, the exterior angles of the edges of $\Pi_3$ are not even rational numbers, and hence are different from the formulas we derived in Theorem \ref{thm:uniqueness}.

\bibliography{biblio}
\bibliographystyle{plain}

\end{document}